\documentclass[11pt]{article}
\usepackage{showlabels}
\usepackage{calc}
\usepackage[margin=2.5cm]{geometry}
\usepackage{color}
\usepackage{amsmath,amssymb,amsthm}
\usepackage{fancyhdr}
\usepackage{psfrag}
\usepackage{array}
\usepackage{graphicx}
\usepackage{hyperref}
\usepackage{bbold}
\usepackage{enumitem}

\newtheorem{theorem}{Theorem}
\newtheorem{lemma}[theorem]{Lemma}
\newtheorem{proposition}[theorem]{Proposition}

\newtheorem{definition}[theorem]{Definition}

\newtheorem{remark}[theorem]{Remark}

\newtheorem{example}[theorem]{Example}
\def\dst{\displaystyle}
\def\eps{\varepsilon}
\def\supp{\text{supp}}

\def\p{\partial}
\def\a{\alpha}
\def\R{\mathbb{R}}
\def\C{\mathcal{C}}
\def\B{\mathcal{B}}
\def\D{\mathcal{D}}
\def\M{\mathcal{M}}
\def\W{\mathbb{W}}
\def\N{\mathbb{N}}
\def\div{{\mbox{\small\rm  div}}\,}

\date{}
\newcommand{\comMagali}[1]{\textcolor{black}{#1}}

\newcommand{\MT}[1]{\textcolor{black}{#1}}

\newcommand{\rev}[1]{\textcolor{black}{#1}}
\title{A Wasserstein norm for signed measures, with application to non local transport equation with source term}
\author{
Benedetto Piccoli
\thanks{Department of Mathematical Sciences, Rutgers University Camden, Camden, NJ, USA. Email: piccoli@camden.rutgers.edu.
}
\and Francesco Rossi 
\thanks{
Dipartimento di Matematica ``Tullio Levi-Civita''
Universit\`a  degli Studi di Padova, Via Trieste 63, 35121 Padova, Italy. Email: francesco.rossi@math.unipd.it}
\and Magali Tournus 
\thanks{Aix Marseille Univ, CNRS, Centrale Marseille, I2M, Marseille, France. Email: mtournus@math.cnrs.fr.}
}

\begin{document}
\maketitle
\begin{abstract}
 We introduce the optimal transportation interpretation of the Kantorovich norm on the space of signed Radon measures with finite mass,
based on a generalized Wasserstein distance 
 for measures with different masses.

With the formulation and the new topological properties we obtain for this norm, we prove existence and uniqueness for solutions to non-local transport equations with source terms, when the initial condition is a signed measure.
\end{abstract}

 {\bf Keywords.} Wasserstein distance, Transport equation, Signed measures, Kantorovich duality.
 
 \medskip 
 
 {\bf AMS subject classifications.} 28A33, 35A01.

\section{Introduction}

The problem of optimal transportation, also called  Monge-Kantorovich problem, has been intensively studied in the mathematical community. Related to this problem, Wasserstein distances in the space of probability measures have revealed to be powerful tools, in particular for dealing with dynamics of measures like the transport Partial Differential Equation (PDE in the following), see e.g. \cite{AG08,AGS08}. For a complete introduction to Wasserstein distances, see \cite{Villani_topics,Villani_old_new}.

The main limit of this approach, at least for its application to dynamics of measures, is that the Wasserstein distances $W_p(\mu,\nu)$ ($p\geq 1$) are defined only if the two positive measures $\mu,\nu$ have the same mass. For this reason, \rev{the generalized Wasserstein distances $W_p^{a,b}(\mu,\nu)$ are introduced in \cite{PR14,PR16}: they combine the standard Wasserstein and total variation distances.} In rough words, for $W_p^{a,b}(\mu,\nu)$ an infinitesimal mass $\delta\mu$ of $\mu$ can either be removed at cost $a |\delta\mu|$, or moved from $\mu$ to $\nu$ at cost $bW_p(\delta\mu,\delta\nu)$. 
\rev{An optimal transportation problem between densities with different masses 
has been studied in \cite{CMC10,F10} where only a given fraction $m$ 
of each density is transported.
These works were motivated by a modeling issue:  
using the example of a resource that is extracted and that we want to distribute in factories, 
 one aims to use only a
certain given fraction of production and consumption capacity.
In this approach and contrarily to the generalized Wasserstein distance \cite{PR12}, the mass that is leftover has no impact on the distance between the measures $\mu$ and $\nu$.
In another context, for the purpose 
to interpret some reaction-diffusion equations not preserving masses as gradient flows, the authors of \cite{FG10} define the
distance $Wb_2$ between measures with different masses on a bounded domain.}
Further generalizations for positive measures with different masses, based on the Wasserstein distance, are introduced in \cite{Chizat2016,KMV16,LMS16}.

Such generalizations still have a drawback: both measures need to be positive. The first contribution of this paper is then the definition of a norm 
on the space of signed Radon measures with finite mass on $\R^d$. Such norm, based on an optimal transport approach, induces a distance generalizing the Wasserstein distance to signed measures.
We then prove that this norm corresponds to the extension of the so-called Kantorovich distance for finite signed Radon measures introduced in \cite{Hanin99} in 
the dual form \begin{equation}\label{Hanin}
\|\mu\|= \sup_{\|f\|_{\infty} \leq 1, \; \|f\|_{Lip}\leq 1}\dst\int_{\R^d}fd\mu.               
              \end{equation}
The novelty then  lies in the dual interpretation of this norm in the framework of optimal transportation.
We also prove new topological properties and characterizations of this norm.

The second main contribution of the paper is to use  this norm to guarantee well-posedness of the following non local transport equation 
with a source term being a signed measure.
We study the following PDE
 \begin{equation}
 \label{transport_equation}
  \p_t \mu_t(x) + \div(v[\mu_t](x) \mu_t(x)) = h[\mu_t](x), \qquad   \mu_{|t=0}(x)=\mu_0(x),
 \end{equation}
 for $x \in \R^d$ and $\mu_0 \in \M^s(\R^d)$,
 where $\M^s(\R^d)$ is the space of signed Radon measures with finite mass on $\R^d$.
Equation \eqref{transport_equation} has already been studied in the framework of positive measures, where 
it has been used for modeling several different phenomena such as crowd motion and  development in biology, see a review in \cite{PR20}.
 \rev{From the modeling point of view,}
one of the interests of signed measures is that they can be used 
to model phenomena \rev{for which the measures under study are intrinsically signed}.
\rev{For instance, in a model coming from the hydrodynamic equations of Ginzburg-Landau vortices, 
the vortex density  $\mu_t$  (which can be positive or negative depending the local topological degree) in domain occupied by a superconducting sample satisfies \eqref{transport_equation} with 
$h[\mu_t]=0$ and where
$v[\mu_t]$ is the magnetic field induced in the sample (see \cite{AMS11} and \cite{Mainini11}).}

\rev{Another} motivation to study equation \eqref{transport_equation} in the framework of signed measure 
 is the interpretation of $\mu_t$ as the spatial derivative of the entropy solution $\rho(x,t)$ to a scalar conservation law. 
 A link between scalar conservation laws and non local transport equation has been initiated in \cite{BCdFP15,JV16}, 
 but until now, studies are restricted to convex fluxes and 
 monotonous initial conditions, so that the spatial derivative $\mu_t$ is a positive measure for all $t>0$.
 \rev{To deal with generic scalar conservation laws, one need a space of signed measures equipped
 with a metric of Wasserstein type, see e.g. \cite{BBL05}.}
\\

The authors of \cite{AMS11} suggested to extend the usual Wasserstein distance $W_1$ to the couples of signed measures $\mu= \mu^+ - \mu^-$ and $\nu= \nu^+ - \nu^-$ 
such that $|\mu^+|+ |\nu^-|= |\mu^-|+ |\nu^+|$ 
by the formula $\W_1(\mu,\nu)= \W_1(\mu^+ + \nu^-, \mu^- + \nu^+)$.
This procedure fails for $p\neq1$, since triangular inequality is lost.
\rev{A counter-example to the triangular inequality is provided in  \cite{AMS11} for $d=1$ and  $p=2$: taking $\mu=\delta_0$, 
$\nu=\delta_4$, $\eta=\delta_1-\delta_2+\delta_3$, 
we obtain $\W_2(\mu,\nu)= 4$ whereas $\W_2(\mu,\eta)+\W_2(\eta,\nu) =\sqrt{2}+\sqrt{2}$}.

We use the same trick from \cite{AMS11} to turn the generalized Wasserstein distance $W_1^{a,b}$ 
into a distance for signed measures,, by setting
$\W_1^{a,b}(\mu,\nu)$ as $W_1^{a,b}(\mu^+ + \nu^-, \mu^- + \nu^+)$. 
\rev{For the same reason as mentioned above, 
this construction cannot be done for $p\neq 1$, in particular, no quadratic distance can 
be obtained with this construction.}
The space of signed measures being a vector space, we also define a norm 
$\|\mu\|^{a,b}= \W_1^{a,b}(\mu,0)$.
Notice that to  define the norm $\|.\|^{a,b}$, we need to restrict ourselves to Radon measures with finite mass, since the generalized Wasserstein distance \cite{PR14}
may not be defined for Radon measures with infinite mass.
\rev{Since the terms ``Wasserstein'' and ``norm'' are usually not used together, 
we emphasize that $\|.\|^{a,b}$ is a norm in the sense of linear vector spaces.}\\
We then use the norm $\|\cdot\|^{a,b}$ to study existence and uniqueness of the solution to the equation \eqref{transport_equation}. The regularity assumptions made in this paper on the vector field and on the source term are the 
 following: \\ 

\begin{enumerate}[label=(H-\arabic*)]
\item \label{H1} 
There exists $K$ such that for all $\mu , \nu \in \mathcal{M}^s(\R^d)$ it holds
\begin{equation}\label{Hyp_v:1}
 \|v[\mu]- v[\nu]\|_{\C^0(\R^d)} \leq K \|\mu- \nu\|^{a,b}.
\end{equation}

  \item \label{H2} 
 There exist $L,M$ such that for all $x,y \in \R^d$, for all $\mu \in \mathcal{M}^s(\R^d)$ it holds
\begin{equation}\label{Hyp_v:2}
 |v[\mu](x)- v[\mu](y)|  \leq L|x-y|, \qquad |v[\mu](x)| \leq M.
\end{equation}
 \item \label{H3} 
 There exist $Q,P,R$ such that for all $\mu, \nu \in \mathcal{M}^s(\R^d)$ it holds
 
\begin{equation}\label{Hyp_h:1}
\|h[\mu] -  h[\nu]\|^{a,b} \leq Q \|\mu- \nu\|^{a,b}, \qquad \left|h[\mu]\right| \leq P, \qquad \supp(h[\mu]) \subset B_0(R).
\end{equation}
\end{enumerate}
\rev{Hypothesis \ref{H1} guarantees that 
$v[\mu_t]$ is continuous in space, and then
the product $v[\mu_t]. \mu_t$ is well-defined.}
The main result about equation \eqref{transport_equation} is the following:
\begin{theorem}[Existence and uniqueness]\label{t-exun}
 Let $v$ and $h$ satisfy \ref{H1}-\ref{H2}-\ref{H3} and $\mu_0 \in \M^s(\R^d)$ compactly supported be given. 
 Then, there exists a unique distributional solution to \eqref{transport_equation} in the space $\C^0\left([0,1], \M^s(\R^d)\right)$
\rev{equipped with $\|\mu_t\|=\sup_{t\in[0,1]}\|\mu_t\|^{a,b}$ }. In addition, for $\mu_0$ and $\nu_0$ in $\M^s(\R^d)$, denoting by $\mu_t$ and $\nu_t$ the corresponding solutions, 
we have the following property for $t\in [0,1)$ of continuous dependence with respect to initial data:
\begin{equation*}
 \|\mu_t-\nu_t\|^{a,b} \leq  \|\mu_0-\nu_0\|^{a,b}  \exp(C_1t), \qquad C_1= 2L + 2K(P +\min\{|\mu_0|,|\nu_0|\})+Q, 
\end{equation*}
\rev{the following estimates on the mass and support:
\begin{equation*}
 |\mu_t|\leq |\mu_0|+P t, \qquad \supp\{\mu_t\} \subset B(0,R'+tM)\; \text{ for } R' \text{ such that }
\left(\supp\{\mu_0\}\cup B_0(R)\right)\subset B_0(R'),
\end{equation*}
the solution is Lipschitz in time:
\begin{equation*}
 \|\mu_{t+\tau}-\mu_t\|^{a,b} \leq C_2 \tau, \quad C_2= P+bM(P(t+\tau)+|\mu_0|),
 \quad  \tau \geq 0.
\end{equation*}
}
 \end{theorem}
\rev{A precise definition of measure-valued weak solution for equation \eqref{transport_equation} is provided at the beginning 
of Section \ref{sec:transport_equation}.}
\begin{remark}
We emphasize that the assumptions \ref{H2}-\ref{H3} are incompatible with a direct interpretation of 
the solution of \eqref{transport_equation} as the spatial derivative of a conservation law
and need to be relaxed in a future work.
Indeed, to draw a parallel between conservation laws and non-local equations,
discontinuous vector fields need to be considered. 
\end{remark}

The structure of the article is the following. In Section \ref{sec:preliminary}, we state and prove preliminary results which are needed for the 
rest of the paper. 
In Section \ref{sec:results}, we define the generalized Wasserstein distance for signed measures, we
show that it can be used to define a norm, and prove some topological properties.
Section \ref{sec:transport_equation} is devoted to the use of the norm defined here to guarantee existence, 
uniqueness, and stability to initial condition for the transport equation \eqref{transport_equation}.

\section{Measure theory and the Generalized Wasserstein distance}
\label{sec:preliminary}

In this section, we introduce the notations and state preliminary results. 
Throughout the paper, $\B(\R^d)$ is the space of Borel sets on $\R^d$, 
$\M(\R^d)$ is the space of Radon measures with finite mass (i.e. Borel regular, positive, and finite on every set).

\subsection{Reminders on measure theory}
In this section, $\mu$ and $\nu$ are in $\M(\R^d)$.
\begin{definition}We say that 
 \begin{itemize}
  \item $\mu << \nu$ if $\forall A \in \B(\R^d),\; (\nu(A) = 0) \Rightarrow (\mu(A)=0)$
  \item $\mu \leq \nu $ if $\forall A \in \B(\R^d), \; \mu(A) \leq \nu(A) $
 \item $\mu \perp \nu$ if there exists $E \in \B(\R^d)$ such that 
 $\mu(\R^d)= \mu( E)$ and $ \nu( E^c)=0$
 \end{itemize}

\end{definition}

The concept of largest common measure between measures is now recalled.

\begin{lemma}
 We consider $\mu$ and $\nu$ two measures in $\M(\R^d)$. Then, there exists  a unique 
 measure $\mu \wedge \nu$ which satisfies
 \begin{equation}
 \label{mu_wedge_nu_def}
  \mu \wedge \nu \leq \mu, \quad   \mu \wedge \nu \leq \nu, \quad (\eta \leq \mu \text{ and } \eta\leq \nu) \Rightarrow \eta \leq \mu \wedge \nu.
 \end{equation}
We refer to $  \mu \wedge \nu$ as the largest common measure to $\mu$ and $\nu$.
Moreover, denoting by $f$ the Radon Nikodym derivative of $\mu$ with respect to $\nu$, 
 i.e. the unique measurable function $f$ such that $\mu= f \nu + \nu_{\perp}$, with $\nu_{\perp}\perp \nu$,
 we have
 \begin{equation}\label{mu_wedge_nu}
  \mu \wedge \nu = \min\{f,1\} \nu. 
 \end{equation}
\end{lemma}
\begin{proof}
 The uniqueness is clear using \eqref{mu_wedge_nu_def}. 
Existence is given by formula \eqref{mu_wedge_nu} as follows.
 First, it is obvious that $ \min\{f,1\} \nu \leq \nu $ and using $\mu= f \nu + \nu_{\perp}$, it is also clear that 
 $  \min\{f,1\} \nu \leq \mu $.
 Let us now assume by contradiction the existence of a measure $\eta$ and of $A \in \B(\R^d)$ such that
 \begin{equation}
 \label{absurd1}
 \eta\leq \mu, \quad \eta\leq \nu, \quad \eta(A) > \dst\int_A \min\{f,1\} d\nu.
 \end{equation}

Since $\nu_{\perp} \perp \nu$, there exists $E \in \B(\R^d)$ such that $\nu(A)= \nu(A\cap E)$ 
and $ \nu_{\perp}(A) = \nu_{\perp}(A\cap E^c)$.
Since $\eta \leq\nu$, we have 
\begin{equation*}
\eta(A\cap E) =\eta(A)  > \dst\int_{A\cap E} \min\{f,1\} d\nu.
\end{equation*}
We define
\begin{equation*}
 B = A\cap E\cap\{ f >1\}.
\end{equation*}
If $\nu(B)=0$, then $f\leq 1$ $\nu$-a.e., hence $\eta(A)\leq \dst\int_A \min\{f,1\} d\nu$. We then assume $\nu(B)>0$. Then
\begin{equation*}
 \begin{aligned}
&\eta(B)+ \eta((A\cap E) \setminus B) = \eta(A\cap E)> \dst\int_{B} \min\{f,1\} d\nu(x) + 
\dst\int_{(A\cap E) \setminus B} \min\{f,1\} d\nu \\
& = \int_B 1 d\nu+\int_{(A\cap E) \setminus B} f d\nu = \nu(B) + \mu((A\cap E) \setminus B) \end{aligned}
\end{equation*}
which contradicts the fact that $\eta\leq \nu$ and $\eta \leq \mu$.
This implies that $\eta$ satisfying \eqref{absurd1} does not exist, and then \eqref{mu_wedge_nu} holds.
\end{proof}

\begin{lemma}\label{lemma:prop_min_measures}
  Let $\mu$ and $\nu$ be two measures in $\M(\R^d)$. Then $\eta \leq \mu+ \nu$ implies $\eta - (\mu\wedge \eta) \leq \nu$.

\end{lemma}

\begin{proof}

Take $A$ a Borel set.  We write $\mu = f \eta + \eta_{\perp},$ with $\eta_{\perp} \perp \eta$.
Then $\eta \wedge \mu= \min\{f,1\} \eta$, and we can write
$$\eta(A) - (\eta\wedge \mu)(A) = \int_A\Big(\max\{1-f,0\} \Big)  d\eta.$$
Define $B= A\cap \{f <1\}$, and $E$ such that
$\eta(A \cap E)= \eta(A)$ and $\eta_{\perp}(A \cap E^c)= \eta_{\perp}(A)$. It then holds,
\begin{equation*}
\eta(A) - (\eta\wedge \mu)(A)= \int_{B\cap E}(1-f) d\eta(x) = \eta(B \cap E) + \eta_{\perp}(B\cap E) - \mu(B \cap E) \leq \nu(B \cap E) \leq \nu(A). 
\end{equation*}
Since this estimate holds for any Borel set $A$, the statement is proved.

\end{proof}
\subsection{Signed measures}

We now introduce signed Radon measures, that are measures  
$\mu$ that can be written as $\mu=\mu_+-\mu_-$ with 
$\mu_+,\mu_- \in \M(\R^d)$. We denote with $\M^s(\R^d)$ the space of such signed Radon measures.

For $\mu\in \M^s(\R^d)$, we define 
$|\mu|= |\mu_+^J|+ |\mu_-^J|$ where $(\mu_+^J,\mu_-^J)$ is the unique 
Jordan decomposition of $\mu$, i.e. $\mu=\mu_+^J-\mu_-^J$ with $\mu_+^J\perp \mu_-^J$. 
Observe that $|\mu|$ is always finite, since $\mu_+^J,\mu_-^J\in  \M(\R^d)$.
\rev{
\begin{definition}[Push-forward]
For $\mu\in\M^s(\R^d)$ and $T:\R^d\to\R^d$ a Borel map, 
the \textit{push-forward} $T\#\mu$ is the measure on $\R^d$ defined by 
$T\#\mu(B)= \mu(T^{-1}(B))$ for any Borel set $B\subset\R^d$.
\end{definition}}
We now remind the definition of tightness for a sequence
in $\M^s(\R^d)$.
\begin{definition}
 A sequence $(\mu_n)_{n\in\N}$ of measures in $\M(\R^d)$
 is tight if for each $\eps>0,$ there is a compact 
set $K \subset \R^d$ such that for all $n\geq 0,$ $\mu_n(\R^d\setminus K)<\eps$.
 A sequence $(\mu_n)_{n\in\N}$ of signed measures of $\M^s(\R^d)$
 is tight if the sequences $(\mu_n^+)_{n\in\N}$ and $(\mu_n^-)_{n\in\N}$ given 
 by the Jordan decomposition are both tight.
\end{definition}
\rev{
For a sequence of probability measures, weak and narrow convergences and are equivalent.
It is not the case for signed measure and we precise here what we call narrow convergence.
In the present paper,  $\C^0(\R^d;\R)$ is the set of continuous functions, 
$\C^0_b(\R^d;\R)$ is the set of bounded continuous functions,
$\C^0_c(\R^d;\R)$ is the set of continuous functions 
 with compact support on $\R^d$, and $\C^0_0(\R^d;\R)$ 
 is the set of continuous functions on $\R^d$ that vanish at infinity.
\begin{definition}[Narrow convergence for signed measures]
  \item A sequence $(\mu_n)_{n\in\N}$ of measures in $\M^s(\R^d)$ is said to converge narrowly to 
  $\mu$ if for all $\varphi \in C^0_b(\R^d;\R)$, $\int_{\R^d} \varphi(x)d\mu_n(x) \to \int_{\R^d} \varphi(x)d\mu(x)$.
\\
  \item A sequence $(\mu_n)_{n\in\N}$ of measures in $\M^s(\R^d)$ is said to converge vaguely to 
 $\mu$ if for all $\varphi \in C_c^0(\R^d;\R)$, $\int_{\R^d} \varphi(x)d\mu_n(x) \to \int_{\R^d} \varphi(x)d\mu(x)$.
\end{definition}
}
\begin{lemma}[\rev{Weak compactness for positive measures}]
\label{lemma:weak_compactness}
Let $\mu_n$ be a sequence of  measures in $\M(\R^d)$ that are uniformly bounded in mass. We can then extract a subsequence $\mu_{\phi(n)}$ such that 
$\mu_{\phi(n)}$ \rev{converges vaguely to}
$\mu$
for some $\mu \in \M(\R^d)$.
    \end{lemma}
\rev{A proof can be found in \cite[Theorem 1.41]{Evans_measures}. Notice that
in \cite{Evans_measures}, vague convergence is called weak convergence. 
In \cite{Hanin99,Villani_topics} however, weak convergence refers to what we define here as narrow convergence.
Notice that if a sequence of positive measures $\mu_n$ converges vaguely to $\mu$ and if $(\mu_n)_n$ is tight, then 
$\mu_n$ converges narrowly to $\mu$.}
%
%

\subsection{Properties of the generalized Wasserstein distance}

In this section, we remind key properties of the generalized Wasserstein distance. 
The usual Wasserstein distance $W_p(\mu,\nu)$ was defined 
between two measures $\mu$ and $\nu$ of same mass $|\mu|= |\nu|$, see more details in \cite{Villani_topics}. 
\rev{\begin{definition}[Transference plan]
           A transference plan between two positive measures of same mass $\mu$ and $\nu$  
is a measure $\pi \in \mathcal{P}(\R^d,\R^d)$ which satisfies 
for all $A,B \in \mathcal{B}(\R^d)$
\begin{equation*}
 \pi(A\times \R^d) = \mu(A), \quad  \pi(\R^d\times B) =\nu(B).
\end{equation*}
     \end{definition}
}
\rev{Note that transference plans are \textit{not} probability measures in general, as their mass is $|\mu|=|\nu|$, the common mass of both marginals.}
We denote by $\Pi(\mu,\nu)$ the set of transference plans between $\mu$ and $\nu$.
The p-Wasserstein distance for positive Radon measures of same mass is defined as
\begin{equation*}
 W_{p}(\mu , \nu) = \left(\min_{\pi \in \Pi(\mu,\nu)} \dst\int_{\R^d \times \R^d}|x-y|^pd \pi (x,y)\right)^{\frac 1p}.
\end{equation*}

It was extended to positive measures having possibly different mass in \cite{PR14,PR16}, 
where the authors introduce the distance $W_p^{a,b}$ on the space $\M(\R^d)$ of Radon measures 
with finite mass.
The formal definition is the following.
\begin{definition}[Generalized Wasserstein distance \cite{PR14}] 
Let $\mu, \nu$ be two positive measures in $\M(\R^d)$.
The generalized Wasserstein distance between $\mu$ and $\nu$ is given by 
\begin{equation}
 W_p^{a,b}(\mu,\nu) = \left(  \inf_{\underset{|\tilde{\mu}|=|\tilde{\nu}|}{\tilde{\mu}, \tilde{\nu} \in \mathcal{M}(\R^d)}} a^p(|\mu - \tilde{\mu}| +|\nu- \tilde{\nu}|)^p+ b^p W_p^p(\tilde{\mu},\tilde{\nu}) \right)^{1/p}.\label{e-wp}
\end{equation}
 \end{definition}
We notice that 
\rev{
\begin{equation}
  W_p^{\lambda a,\lambda b}= \lambda W_p^{a,b}, \quad \lambda >0,
\end{equation}
}
and in particular
\begin{equation}
 \label{scaling}
 W_p^{a,b}= \dfrac{b}{b'} \;W_p^{a',b'}, \quad \text{ for } \dfrac{a}{b}=\dfrac{a'}{b'}.
\end{equation}

\rev{
The following lemma is useful to derive properties for the generalized Wasserstein distance.
\begin{lemma}
\label{lemma:infimum}
The infimum in \eqref{e-wp} is always attained. Moreover, 
there always exists a minimizer that satisfy the additional 
constraint $\tilde\mu\leq \mu,\;\tilde\nu\leq \nu$.  
     \end{lemma}
     The proof can be found in \cite{PR14}.\\
}

For $f\in \C^0_c(\R^d;\R),$ we define
\begin{equation*}
\|f\|_{\infty}=\sup_{x\in\R^d}|f(x)| , \qquad\|f\|_{Lip}=\sup_{x\neq y} \dfrac{|f(x)-f(y)|}{|x-y|}.
\end{equation*}
We also denote by $\C^{0,Lip}_{c}(\R^d;\R)$ the subset of functions $f\in\C^0_c(\R^d;\R)$ for which it holds $\|f\|_{Lip}<+\infty$.

The following result is stated in \cite[Theorem 13]{PR16}.
\begin{lemma}[Kantorovitch Rubinstein duality]
\label{lemma:KR}
 For $\mu,\;\nu$ in $\M(\R^d)$, it holds
\begin{equation*}
 W_1^{1,1}(\mu,\nu) =  \sup \left\{\int_{\R^d}\varphi \; d(\mu-\nu); \; \varphi \in \C^{0,Lip}_{c},  \|\varphi\|_{\infty} \leq 1,\|\varphi\|_{Lip}\leq 1 \right\}.
\end{equation*}
 \end{lemma}

\begin{lemma}[Properties of the generalized Wasserstein distance]
\label{lemma:prop_gen_dist}
Let $\mu, \nu , \eta, \mu_1, \mu_2, \nu_1, \nu_2 $ be some positive measures 
 with finite mass on $\R^d$.
 The following properties hold
 
 \begin{enumerate}

  \item $W_p^{a,b}(\mu_1 + \mu_2 ,\nu_1 +\nu_2) \leq
  W_p^{a,b}(\mu_1  ,\nu_1 ) 
 + W_p^{a,b}(\mu_2 ,\nu_2) $.
 
  \item $W_1^{a,b}(\mu + \eta,\nu +\eta) = W_1^{a,b}(\mu,\nu), $

  \end{enumerate}

\end{lemma}

\begin{proof}
 The first property is taken from \cite[Proposition 11]{PR14}. 
\rev{For $a=b=1$,} the second statement is a direct consequence of the Kantorovitch-Rubinstein duality in Lemma \ref{lemma:KR} for $W^{1,1}$.
For general $a>0,\;b>0$, we proceed as follows.
Let $\mu,\nu$ be two measures. Define
$$C^{a,b}(\bar\mu,\bar\nu,\pi;\mu,\nu):=a(|\mu-\bar\mu|+|\nu-\bar\nu|)+b\int|x-y|\,d\pi(x,y),$$
where $\pi$ is a transference plan in $\Pi(\bar\mu,\bar\nu)$.
Define $D_\lambda:x\to \lambda x$ with $\lambda>0$ the dilation in $\R^n$.
It holds
\begin{equation*}
\begin{aligned}
C^{a,b}(D_\lambda\#\bar\mu,D_\lambda\#\bar\nu,&(D_\lambda\times D_\lambda)\#\pi;D_\lambda\#\mu,D_\lambda\#\nu)\\
&=a(|D_\lambda\#\mu-D_\lambda\#\bar\mu|+|D_\lambda\#\nu-D_\lambda\#\bar\nu|)+b\int|x-y|\,d(D_\lambda\times D_\lambda)\pi(x,y),\\
&=a(|\mu-\bar\mu|+|\nu-\bar\nu|)+b\int|\lambda x-\lambda y|\,d\pi(x,y)=C^{a,\lambda b}(\bar\mu,\bar\nu,\pi;\mu,\nu).
\end{aligned}
\end{equation*}
As a consequence, it holds
$$W^{a,b}(D_\lambda\#\mu,D_\lambda\#\nu)=W^{a,\lambda b}(\mu,\nu).$$

We now show that this implies $W^{a,b}(\mu+\eta,\nu+\eta)=W^{a,b}(\mu,\nu).$
Indeed, also applying Kantorovich-Rubinstein for $W^{1,1}$ and \eqref{scaling} with $a'=1,b'=\lambda=\frac{b}{a}$, it holds
\begin{eqnarray*}
&&W^{a,b}(\mu+\eta,\nu+\eta)=aW^{1,\frac{b}{a}}(\mu+\eta,\nu+\eta)=aW^{1,1}(D_\lambda\#\mu+D_\lambda\#\eta,D_\lambda\#\nu+D_\lambda\#\eta)=\\
&&=aW^{1,1}(D_\lambda\#\mu,D_\lambda\#\nu)=aW^{1,\frac{b}{a}}(\mu,\nu)=W^{a,b}(\mu,\nu).
\end{eqnarray*}
\end{proof}
%
%
 \begin{definition}[Image of a measure under a plan]
 \label{def:image}
    Let $\mu$ and $\nu$ two measures in $\M(\R^d)$ of same mass and $\pi \in \Pi(\mu,\nu)$.
     For $\eta \leq \mu$, we denote by $f$ the Radon-Nikodym derivative of $\eta $ with respect to $\mu$ and 
     by $\pi_f$ the transference plan defined by $\pi_f(x,y) =f(x) \pi (x,y)$. 
     Then, we define the image of $\eta$ under $\pi$ as the second marginal $\eta'$
     of $\pi_f$.
     \end{definition}
Observe that the second marginal satisfies  $\eta'\leq \nu$. Indeed, since $\eta\leq \mu$, it holds $f\leq 1$. Thus, for all Borel set $B$  of $\R^d$ we have
 \begin{equation*}
  \eta'(B) = \pi_f(\R^d\times B) \leq \pi(\R^d\times B) = \nu(B). 
 \end{equation*}

\section{Generalized Wasserstein norm for signed measures}

\label{sec:results}

In this section, we define the generalized Wasserstein distance for signed measures and prove some of its properties. The idea is to follow what was already done in \cite{AMS11} for generalizing the classical Wasserstein distance.

\begin{definition}[Generalized Wasserstein distance extended to signed measures]
 For $\mu, \nu$ two signed measures with finite mass over $\R^d$, we define
 \begin{equation*}
\mathbb{W}_1^{a,b}  (\mu, \nu) = {W}_1^{a,b} (\mu_+ + \nu_-, \mu_- + \nu_+),
 \end{equation*}
 where $\mu_+, \mu_-, \nu_+$ and $\nu_-$ are any measures in $\M(\R^d)$ such that 
 $\mu = \mu_+ - \mu_-$ and $\nu = \nu_+ - \nu_-$.
\label{def:GWDSM}
\end{definition}

\begin{proposition}
 The operator $\mathbb{W}_1^{a,b} $ is a distance on the space $\mathcal{M}^s(\R^d)$
 of signed measures with finite mass on $\R^d$. 
\end{proposition}

\begin{proof}
 First, we point out that the definition does not depend on the decomposition.
Indeed, if we consider two distinct decompositions, 
$\mu = \mu_+ - \mu_- = {\mu}^J_+ - \mu_-^J$, and 
$\nu = \nu_+ - \nu_- = \nu_+^J - {\nu_-}^J$, with the second one being the Jordan decomposition, 
then
we have $(\mu_+ + \nu_-) - ({\mu}_+^J + {\nu}_-^J) = (\mu_- + \nu_+) -({\mu}_-^J + {\nu}_+^J)$, 
and this is a positive measure since $\mu_+ \geq \mu_+^J$ and $\nu_+ \geq \nu_+^J$. 
The \rev{second} property of Lemma \ref{lemma:prop_gen_dist} then gives
\begin{eqnarray*}
&&{W}_1^{a,b} (\mu_+^J + \nu_-^J, \mu_-^J + \nu_+^J)=\\
&& W_1^{a,b}(\mu_+^J + \nu_-^J + (\mu_+ + \nu_-) - ({\mu}_+^J + {\nu}_-^J),\mu_-^J + \nu_+^J + (\mu_- + \nu_+) -({\mu}_-^J + {\nu}_+^J))=\\
&&{W}_1^{a,b} (\mu_+ + \nu_-, \mu_- + \nu_+).
\end{eqnarray*}

We now prove that $\mathbb{W}_1^{a,b}(\mu,\nu)=0$ implies $\mu = \nu$. 
As explained above, we can choose the Jordan decomposition for both $\mu$ and $\nu$.
Since $W_1^{a,b}$ is a distance, we obtain $\mu_+ + \nu_-= \mu_- + \nu_+$.
The orthogonality of $\mu_+$ and $\mu_-$ and of $\nu_+$ and $\nu_-$ 
implies that $\mu_+ = \nu_+$ and $\mu_-=\nu_-$,
and thus $\mu = \nu$.

We now prove the triangle inequality. 
We have $\mathbb{W}_1^{a,b}(\mu,\eta) = W_1^{a,b}(\mu_+ + \eta_-, \mu_- + \eta_+)$.
Using Lemma \ref{lemma:prop_gen_dist}, we have
\begin{equation*}
\begin{aligned}
 \mathbb{W}_1^{a,b}(\mu,\eta)& = {W}_1^{a,b}(\mu_++ \eta_-+ \nu_+ + \nu_-,\mu_-+ \eta_+ + \nu_+ + \nu_-)\\
&\leq {W}_1^{a,b}(\mu_++ \nu_-,\mu_-+ \nu_+ )+ {W}_1^{a,b}( \eta_-+ \nu_+ , \eta_++ \nu_-)\\
&=\mathbb{W}_1^{a,b}(\mu,\nu)+ \mathbb{W}_1^{a,b}(\nu,\eta).
\end{aligned}
\end{equation*}

\end{proof}

We also state the following lemma about adding and removing masses.
\begin{lemma}\label{lemma:prop_gwd}
 Let $\mu, \nu , \eta, \mu_1, \mu_2, \nu_1, \nu_2 $ in $\M^s(\R^d)$
 with finite mass on $\R^d$.
 The following properties hold
 
 \begin{itemize}
  
  \item $\W_1^{a,b}(\mu + \eta,\nu +\eta) = \W_1^{a,b}(\mu,\nu), $

  \item $\W_1^{a,b}(\mu_1 + \mu_2 ,\nu_1 +\nu_2) \leq
  \W_1^{a,b}(\mu_1  ,\nu_1 ) 
 + \W_1^{a,b}(\mu_2 ,\nu_2) $.
  \end{itemize}

\end{lemma}
\begin{proof}
 The proof is direct. For the first item, it holds
  \rev{
  \begin{eqnarray*}
   \W_1^{a,b}(\mu + \eta,\nu +\eta) &=&
 W_1^{a,b}([\mu_+ + \eta_+] + [\nu_- + \eta_-],[\nu_++\eta_+]+ [\mu_-+\eta_-]) \\
&=&W_1^{a,b}(\mu_+ + \nu_- +[\eta_+ +\eta_-],\nu_+ + \mu_- +[\eta_+ +\eta_-]) 
 \end{eqnarray*}
}
 which \rev{by Lemma \ref{lemma:prop_gen_dist} then equals $ W_1^{a,b} (\mu_+ + \nu_- ,\mu_- + \nu_+ ) =\W_1^{a,b}(\mu ,\nu ) $.} 
 
 For the second item, it holds
\begin{equation*}
 \begin{aligned}
  \W_1^{a,b}(\mu_1 + \mu_2 ,\nu_1 +\nu_2) &= W_1^{a,b}(\mu_{1,+} + \mu_{2,+}+ \nu_{1,-} +\nu_{2,-} ,\nu_{1,+} +\nu_{2,+} +\mu_{1,-} + \mu_{2,-})\\
&\leq  W_1^{a,b}(\mu_{1,+} + \nu_{1,-}  ,\nu_{1,+} +\mu_{1,-} ) + W_1^{a,b}(\mu_{2,+}+ \nu_{2,-} ,\nu_{2,+} + \mu_{2,-})\\
&=\W_1^{a,b}(\mu_1  ,\nu_1 )  + \W_1^{a,b}(\mu_2 ,\nu_2),
\end{aligned}
\end{equation*}
 \rev{where the inequality comes from Lemma \ref{lemma:prop_gen_dist}}.
\end{proof}

\begin{definition}
 \label{def:norm}
 For $\mu \in \M^s(\R^d)$ and $a>0,\; b>0$, we define
 \begin{equation*}
  \|\mu\|^{a,b}=\W_1^{a,b}(\mu,0)=W_1^{a,b}(\mu_+,\mu_-),
 \end{equation*}
where $\mu_+$ and $\mu_-$ are any measures of $\M(\R^d)$ such that
$\mu=\mu_+-\mu_-$.
\end{definition}
 It is clear that the definition of $\|\mu\|^{a,b}$ does not depend on the choice of $\mu_+,\mu_-$ as a consequence of the corresponding property for $W_1^{a,b}$.
\begin{proposition}
 The space of signed measures $(\M^s(\R^d),\|.\|^{a,b})$ is a normed vector space. 
\end{proposition}

\begin{proof}
First, we notice that $\|\mu\|^{a,b}=0$ implies that $W_1^{a,b}(\mu_+, \mu_-)=0$, which is 
$\mu_+=\mu_-$ so that $\mu= \mu_+- \mu_-=0$. 
For triangular inequality,
 using the second property of Lemma \ref{lemma:prop_gwd},
  we have that for $\mu,\eta \in \M^s(\R^d)$, 
 \begin{equation*}
 \|\mu+ \eta\|^{a,b}= \W_1^{a,b} (\mu+ \eta,0) \leq \W_1^{a,b} (\mu,0)+\W_1^{a,b} (\eta,0)  =\|\mu\|^{a,b}+\| \eta\|^{a,b}.
 \end{equation*}
Homogeneity is obtained by writing for $\lambda >0,$
$\|\lambda \mu\|^{a,b}=\W_1^{a,b} (\lambda \mu,0)=W_1^{a,b} (\lambda \mu_+,\lambda \mu_-) $
where $\mu= \mu_+ -\mu_-$.
Using Lemma \ref{lemma:KR}, we have 
\begin{equation*}
\begin{aligned}
     W_1^{a,b} (\lambda \mu_+,\lambda \mu_-) &= 
     \sup \left\{\int_{\R^d}\varphi \; d(\lambda\mu_+-\lambda\mu_-); \; \varphi \in \C^{0,Lip}_{c},  \|\varphi\|_{\infty} \leq 1,\|\varphi\|_{Lip}\leq 1 \right\}
\\
&= \lambda \sup \left\{\int_{\R^d}\varphi \; d(\mu_+-\mu_-); \; \varphi \in \C^{0,Lip}_{c},  \|\varphi\|_{\infty} \leq 1,\|\varphi\|_{Lip}\leq 1 \right\} 
 =\lambda W_1^{a,b} ( \mu_+,\mu_-).
  \end{aligned}
 \end{equation*}

 \end{proof}
 
 \rev{
 We provide here an example that illustrates 
 the competition between cancellation and transportation. 
 This example is  used later in the paper. 
 \begin{example}
 \label{ex}
  Take $\mu= \delta_x-\delta_y$.
  Then
  \begin{equation*}
 \|\mu\|^{a,b}= \W_1^{a,b}(\delta_x-\delta_{y},0)=
 W_1^{a,b}(\delta_x,\delta_{y})=   \inf_{\underset{|\tilde{\mu}|=|\tilde{\nu}|}{\tilde{\mu}, \tilde{\nu} \in \mathcal{M}(\R^d)}} \left\{a(|\delta_x - \tilde{\mu}| +|\delta_y- \tilde{\nu}|)+ b W_1(\tilde{\mu},\tilde{\nu})\right\}. 
\end{equation*}
Using Lemma \ref{lemma:infimum}, the minimum is attained and it can be written as 
$ \tilde{\mu}= \epsilon \delta_x$ and $\tilde{\nu} = \epsilon \delta_y $ for some $0 \leq \epsilon \leq 1$.
Then 
  \begin{equation*}
 \|\mu\|^{a,b}= \min_{0 \leq \epsilon \leq 1 } \left\{2a(1-\epsilon)+ b \epsilon |x-y|\right\}. 
\end{equation*}
The expression above depends on the distance between the Dirac masses $\delta_x$ and $\delta_y$.
For $b |x-y |< 2 a$, then the minimum is attained for $\epsilon=1$ and 
$ \|\mu\|^{a,b}= b|x-y|$. In that case, we say that all the mass is transported.
On the contrary, for $b |x-y |\geq  2 a$, then the minimum is attained for $\epsilon=0$ 
 and 
$ \|\mu\|^{a,b}= 2a$, and we say that all the mass is cancelled (or removed).
\end{example}
}
\subsection{Topological properties}
In this section, we study the topological properties of the norm introduced above.
In particular, we aim to prove that it admits a duality formula that indeed coincides with \eqref{Hanin}.
We first prove that the topology of $\|.\|^{a,b}$ does not depend on $a,b>0.$
\begin{proposition}
For $a>0,\; b>0$, the norm $\|.\|^{a,b}$ is equivalent to $\|.\|^{1,1}$.
\end{proposition}

\begin{proof}
 For $\mu \in \M^s(\R^d)$ 
 denote by $(m^{a,b}_+,m^{a,b}_-)$ the positive measures such that
 \begin{equation*}
  \|\mu\|^{a,b} = a|\mu_+ - m_+^{a,b}| + a|\mu_- - m_-^{a,b}| +bW_1(m_+^{a,b}, m_-^{a,b}),
 \end{equation*}
 and similarly define $(m^{1,1}_+,m^{1,1}_-)$.
 \rev{Their existence is guaranteed by Lemma \ref{lemma:infimum}.} By definition of the minimizers, we have 
\begin{equation*}
\begin{aligned}
  \|\mu\|^{a,b} &= a|\mu_+ - m_+^{a,b}| + a|\mu_- - m_-^{a,b}| +bW_1(m_+^{a,b}, m_-^{a,b}) 
 \\&\leq a|\mu_+ - m_+^{1,1}| + a|\mu_- - m_-^{1,1}| +bW_1(m_+^{1,1}, m_-^{1,1}).
 \leq \max\{a,b\}  \|\mu\|^{1,1},
\end{aligned}
\end{equation*}
In the same way, we obtain 
\begin{equation*}
 \min\{a,b\}  \|\mu\|^{1,1} \leq \|\mu\|^{a,b}\leq \max\{a,b\}  \|\mu\|^{1,1}.
\end{equation*}

\end{proof}

We give now an equivalent Kantorovich-Rubinstein duality formula for the new distance.
For $f\in \C^0_b(\R^d;\R),$ similarly to $\C^0_c(\R^d;\R)$, we define the following
\begin{equation*}
\|f\|_{\infty}=\sup_{x\in\R^d}|f(x)| , \qquad\|f\|_{Lip}=\sup_{x\neq y} \dfrac{|f(x)-f(y)|}{|x-y|}.
\end{equation*}

We introduce 
\begin{equation*}
\C_{b}^{0,Lip} = \{f \in \C^0_{b}(\R^d;\R) \; |\; \|f\|_{Lip} < \infty\}.          
 \end{equation*}
In the next proposition, we express the Kantorovich duality for the norm $\W_1^{1,1}$.
This shows that  $\W_1^{1,1}$ coincides with the bounded Lipschitz distance introduced in \cite{Hanin99}, also called Fortet Mourier distance in \cite{Villani_old_new}.

\begin{proposition}[{Kantorovich duality}]
\label{prop:flat_metric}
The signed generalized Wasserstein distance $\W_1^{1,1}$ 
coincides with the bounded Lipschitz  distance:
 for $\mu,\;\nu$ in $\M^s(\R^d)$, it holds
\begin{equation*}
 \W_1^{1,1}(\mu,\nu) =  \sup \left\{\int_{\R^d}\varphi \; d(\mu-\nu); \; \varphi \in \C^{0,Lip}_{\textcolor{black}{b}},  \|\varphi\|_{\infty} \leq 1,\|\varphi\|_{Lip}\leq 1 \right\}
\end{equation*}
\end{proposition}
\rev{We emphasize that Proposition \ref{prop:flat_metric} does not coincide with Lemma \ref{lemma:KR}, 
since it involves non-compactly supported test functions.}

\begin{proof}
By using Lemma \ref{lemma:KR}
we have
 \begin{equation*}
 \begin{aligned}
   \W_1^{1,1}(\mu,\nu)&=
   W_1^{a,b}(\mu_+ + \nu_-,\nu_++\mu_-)
 \\  &= \sup \left\{\int_{\R^d}\varphi \; d(\mu_+-\mu_--(\nu_+-\nu_-)); \; \varphi \in \C^{0,Lip}_c,  \|\varphi\|_{\infty} \leq 1,\|\varphi\|_{Lip}\leq 1 \right\}
 \\&= \sup \left\{\int_{\R^d}\varphi \; d(\mu-\nu); \; \varphi \in \C^{0,Lip}_c,  \|\varphi\|_{\infty} \leq 1,\|\varphi\|_{Lip}\leq 1 \right\}.
 \end{aligned}
 \end{equation*}

We denote by 
\begin{equation*}
 S= \sup \left\{\int_{\R^d}\varphi \; d(\mu-\nu); \; \varphi \in \C^{0,Lip}_b,  \|\varphi\|_{\infty} \leq 1,\|\varphi\|_{Lip}\leq 1 \right\}.
\end{equation*}
First observe that $S<+\infty$. Indeed, it holds $\int_{\R^d}\varphi \; d(\mu-\nu)\leq\|\varphi\|_\infty(|\mu|+|\nu|)<+\infty$. Denote with  $\varphi_n$ a sequence of functions of $\C^{0,Lip}_b$ such that $\int_{\R^d}\varphi_n \; d(\mu-\nu) \to S$ as $n\to \infty$.
\MT{Consider a sequence of functions $\rho_n$ in $\C^{0,Lip}_c$ such that $\rho_n(x)=1$ for $x\in B_0(n)$, $\rho_n(x)=0$ for $x \notin B_0(n+1)$ and $\|\rho_n\|_{\infty} \leq 1$.}
For the sequence $\psi_n = \varphi_n \MT{\rho_n}$ of functions
of $\C^{0,Lip}_c$, it holds
\begin{equation*}
\begin{aligned}
 \left|\int_{\R^d}\psi_n \; d(\mu-\nu) -S \right|
 &\leq
  \left|\int_{\R^d}\left(\psi_n-\varphi_n\right) \; d(\mu-\nu)  \right|+  \left|\int_{\R^d}\varphi_n \; d(\mu-\nu) -S \right| \\
   &\leq
 2 \left|\int_{\R^d\MT{\setminus B_0(n)}} \; d(\mu-\nu)  \right|+  \left|\int_{\R^d}\varphi_n \; d(\mu-\nu) -S \right| \\
\end{aligned}
\end{equation*}
since $\|\varphi_n\|_{\infty} \leq1$.
The first term goes to zero with $n$, since $(\mu-\nu)$ being of finite mass is tight, and
the second term goes to zero with $n$ by definition of $S$ and $\varphi_n$.
Then
\begin{equation*}
 S=\sup \left\{\int_{\R^d}\varphi \; d(\mu-\nu); \; \varphi \in \C^{0,Lip}_c,  \|\varphi\|_{\infty} \leq 1,\|\varphi\|_{Lip}\leq 1 \right\},
\end{equation*}
and Proposition \ref{prop:flat_metric} is proved.
\end{proof}

\begin{remark}
We observe that a sequence $\mu_n$ of $\M^s(\R)$ which satisfies $\|\mu_n\|^{a,b}  \underset{n \to \infty}{\rightarrow}0$
 is not necessarily tight, and its mass is not necessarily bounded. 
 For instance, we have that
  \begin{equation*}
\nu_n =  \delta_{n} -  \delta_{n+\frac 1n}  
 \end{equation*}
is not tight, whereas it satisfies for $n$ sufficiently large 
\begin{equation*}
 \|\nu_n\|^{a,b}=\frac{b}{n}\underset{n \to \infty}{\rightarrow}0.
\end{equation*}
\rev{
See Example \ref{ex} for the details of the calculation.}
Now take the sequence
 \begin{equation*}
\mu_n = n\; \delta_{\frac{1}{n^2}} - n\; \delta_{-\frac{1}{n^2}}.  
 \end{equation*}
\rev{As explained in Example \ref{ex}, 
depending on the sign of $2a-\frac{2b}{n^2}$, 
we either  cancel the mass or  transport it. For $n$ large enough, 
$2a \geq  \frac{2b}{n^2}$, so we transport the mass.
Thus  for $n$ sufficiently large
\begin{equation*}
\|\mu_n\|^{a,b}= \dfrac{2bn}{n^2}
\underset{n \to \infty}{\rightarrow}0
\end{equation*}
 whereas $|\mu_n|= 2n$ is not bounded. }
\end{remark}

\begin{remark}
\label{weak_convergence}
Norm $\|.\|^{1,1}$ does not metrize \rev{narrow} convergence, contrarily to what is stated in \cite{Hanin99}.
Indeed, take $\mu_n= \delta_{\sqrt{2\pi n+ \frac{\pi}{2}}} -\delta_{\sqrt{2\pi n+ \frac{3\pi}{2}}}$.
We have 
\begin{equation*}
 \|\mu_n\|^{1,1}\leq \left|\sqrt{2\pi n+ \frac{\pi}{2}}-\sqrt{2\pi n+ \frac{3\pi}{2}}\right| \underset{n\to\infty}{\to} 0, 
\end{equation*}
even though for $\varphi(x)=\sin(x^2)$ in $\C_b^{0}(\R)$, we have
\begin{equation*}
 \int_{\R} \varphi d\mu_n  =2, \quad n\in \N.
\end{equation*}

\end{remark}
\begin{remark}
\label{rem:flat_metric}
We  have as a direct consequence of Proposition \ref{prop:flat_metric} that
 \begin{equation}
 \label{weak_conv}
 \|\mu_n-\mu\|^{a,b} \underset{n \to \infty}{\rightarrow}0 \quad \quad 
 \Rightarrow  \quad \forall \varphi \in \C_{b}^{0,Lip}(\R^d), \; \int_{\R^d} \varphi d\mu_n  \underset{n \to \infty}{\rightarrow}\int_{\R^d} \varphi d\mu. 
 \end{equation}
However, the reciprocal statement of \eqref{weak_conv} is false: 
define 
\begin{equation*}
\mu_n:=n\cos(nx)\chi_{[0,\pi]} .
\end{equation*} 
For $$\varphi_n:=\frac{1}{n}\cos(nx),$$

it is clear that \begin{equation*}
                  \int_{\R} \varphi_n\,d\mu_n=\int_0^{\pi} \cos^2(nx)\, dx=\dfrac{\pi}{2}\not\to 0.
                 \end{equation*}
\rev{In particular,
\begin{equation*}
 \sup_{\varphi \in \C^{0,Lip}_b(\R)} \int_{\R} \varphi\,d(\mu_n-0)\geq \frac{\pi}{2}, 
\end{equation*}
hence by Proposition \ref{prop:flat_metric}, $\|\mu_n-0\|\geq \frac{\pi}{2}$ does not converge to zero.}
We now prove that, for each $\varphi$ in $\C_{b}^{0, Lip}(\R)$, 
it holds $\int_{\R} \varphi\,d\mu_n\to 0$. Given $\varphi \in \C_{b}^{0,Lip}(\R)$, define
$$f(x):=\begin{cases}
\varphi(-x),&\mbox{~~when~~}x\in[-\pi,0],\\
\varphi(x),&\mbox{~~when~~}x\in[0,\pi],\\
\end{cases}$$
and we extend $f$ as a $2\pi$-periodic function on $\R$.
We have
\begin{equation*}
 \dst\int_{\R} \varphi\,d\mu_n=\int_{\R} f \,d\mu_n.
\end{equation*}
 Since $f$ is a $2\pi$-periodic function, it also holds $\int f \,d\mu_n=na_n$, where $a_n$ is the $n$-th cosine coefficient in the Fourier series expansion of $f$.
We then prove $na_n\to 0$ for any $2\pi$-periodic Lipschitz function $f$,
following the ideas of \cite[p. 46, last line]{Zygmund_book}.
Since $f$ is Lipschitz,
then its distributional derivative is in $L^{\infty}[-\pi,\pi]$ and thus in $L^1[-\pi,\pi]$.
Then 
$$a_n=\frac{1}{2\pi}\int_{-\pi}^{\pi} f(x)\cos(nx)\,dx=-\frac{1}{2n\pi}\int_{-\pi}^{\pi} f'(x)\sin(nx)\,dx=-\frac{b'_n}{n},$$
where $b'_n$ is the $n$-th sine coefficient of $f'$.
As a consequence of the Riemann-Lebesgue lemma, $b'_n\to 0$, and this implies $n a_n\to 0$.

\end{remark}

 \begin{proposition}
  Assume that $\|\mu_n\|^{a,b}\underset{n \to \infty}{\to}0$, then  
 $\Delta m_n:= |\mu_n^+|-|\mu_n^-|\underset{n \to \infty}{\to}0$.
 \end{proposition}

 \begin{proof}
  We have by definition $\|\mu_n\|^{a,b}=W_1^{a,b}(\mu_n^+,\mu_n^-)$.
We denote by $\bar{\mu}_n^+, \bar{\mu}_n^-$ the minimizers in the right hand side of \eqref{e-wp} realizing the distance $W_1^{a,b}(\mu_n^+,\mu_n^-)$.
We have 
\begin{equation*}
 \|\mu_n\|^{a,b}=a \left(|\mu_n^+ -\bar{\mu}_n^+|+ |\mu_n^- -\bar{\mu}_n^-| \right) + bW_1(\bar{\mu}_n^+,\bar{\mu}_n^-), \qquad |\bar{\mu}_n^+|=|\bar{\mu}_n^-|.
\end{equation*}

Since $\|\mu_n\|^{a,b}\underset{n \to \infty}{\to}0$, each of the three terms converges to zero as well. 
\rev{Since by Lemma \ref{lemma:infimum} we can assume $\bar{\mu}_n^+\leq {\mu}_n^+$ and $\bar{\mu}_n^-\leq {\mu}_n^-$, we have}
\begin{equation*}
\begin{aligned}
 \left| |\mu_n^+|- |\mu_n^-|\right|&= 
 \left| |\mu_n^+-\bar{\mu}_n^+ +\bar{\mu}_n^+ |- |\mu_n^--\bar{\mu}_n^- +\bar{\mu}_n^-|\right|\\
 & =\left| |\mu_n^+-\bar{\mu}_n^+| +|\bar{\mu}_n^+ |- |\mu_n^--\bar{\mu}_n^-|-|\bar{\mu}_n^-|\right|
\\& = \left| |\mu_n^+-\bar{\mu}_n^+| -|\mu_n^--\bar{\mu}_n^-| \right|
 \underset{n \to \infty}{\to}0.
 \end{aligned}
\end{equation*}

\end{proof}

%
 
We remind from \cite{PR16} that the space $(\M(\R^d),W_p^{a,b})$ is a complete metric space. 
  The proof is based on the fact that 
a Cauchy sequence of positive measures is both uniformly bounded in mass and tight. 
This is not true anymore for a Cauchy sequence of signed measures.

 \begin{remark}
  \label{prop:notaBanach} Observe that $(\M^s(\R^d),\|.\|^{a,b})$ is not a Banach space. Indeed, take the sequence 
  \begin{equation*}
   \mu_n= \sum_{i=1}^n \left( \delta_{i+\frac{1}{2^i}} - \delta_{i-\frac{1}{2^i}}\right).
  \end{equation*}
It is a Cauchy sequence in $(\M^s(\R^d),\|.\|^{a,b})$: indeed, 
\rev{by choosing to transport all the mass from $\mu_n^++\mu_{n+k}^-$ onto $\mu_{n+k}^++\mu_n^-$ with the cost $b$, it holds}
\begin{equation*}
 \W_1^{a,b}(\mu_n, \mu_{n+k}) \leq 2b \sum_{i=n+1}^{n+k} \dfrac{1}{2^i}
 \leq 2b \sum_{i=n+1}^{+\infty} \dfrac{1}{2^i} \underset{n\to \infty}{\to}0.
\end{equation*}

However, the  sequence $(\mu_n)_n$ does not converge in  $(\M^s(\R^d),\|.\|^{a,b})$.
As seen in Remark \ref{rem:flat_metric}, the convergence for the norm $\|.\|^{a,b}$ implies the 
convergence in the sense of distributions. 
In the sense of distributions we have 
\begin{equation*}
\mu_n \rightharpoonup \mu^* :=   \sum_{i=1}^{+ \infty} \left( \delta_{i+\frac{1}{2^i}} - \delta_{i-\frac{1}{2^i}}\right) \notin \M^s(\R).                                    
 \end{equation*}
Indeed, for all $\varphi\in\mathcal{C}^{\infty}_c(\R)$, since $\varphi$ is compactly supported, it holds
\begin{equation*}
 \langle \mu_n - \mu , \varphi\rangle = \sum_{i=n+1}^{+ \infty} \left(\varphi\left(i + \frac{1}{2^i} \right)- \varphi\left(i - \frac{1}{2^i}\right)  \right) 
  \underset{n\to \infty}{\to}0.
\end{equation*}
\rev{The measure $\mu^*$ does not belong go $\M^s(\R)$, as it has infinite mass.}
\end{remark}

Nevertheless, we have the following convergence result. 

 \begin{theorem}
 \label{thm:Banach}
 \rev{Let $\mu_n$ be a Cauchy sequence in  $(\M^s(\R^d),\|.\|^{a,b})$. If $\mu_n$
  is tight and has uniformly bounded  mass, then it converges
 in  $(\M^s(\R^d),\|.\|^{a,b})$.}
 \end{theorem}
 \begin{proof}
 Take a tight Cauchy sequence $(\mu_n)_n \in \M^s(\R^d)$ such that the sequences given by the Jordan decomposition 
 $|\mu_n^+|$ 
 and $|\mu_n^-|$ are 
  uniformly bounded.
 Then, by Lemma \ref{lemma:weak_compactness}, there exists $\mu^+$
 and $\mu^-$ 
 in $\M(\R^d)$ and 
 $\varphi_1$ non decreasing
 such that, $\mu_{\varphi_1(n)}^+ \underset{n\to\infty}{\rightharpoonup} \mu^+$ vaguely.
\rev{Then, $|\mu_n^-|$  being uniformly bounded, there exists  $\varphi_2$ non decreasing such that 
for $\varphi= \varphi_1 \circ \varphi_2$ it holds
$$\mu_{\varphi(n)}^- \underset{n\to\infty}{\rightharpoonup} \mu^-\mbox{~~ vaguely}.$$
}
Since $\mu_n^+$ and $\mu_n^-$ are assumed to be tight, the sequences
$\mu_{\varphi(n)}^- $ and $\mu_{\varphi(n)}^+ $ also converge to $\mu^-$ and $\mu^+$ narrowly,
and it holds
$W_1^{a,b}(\mu_{\varphi(n)}^+,\mu^+) \underset{n\to\infty}{\to} 0$ and 
$W_1^{a,b}(\mu_{\varphi(n)}^-,\mu^-) \underset{n\to\infty}{\to} 0$ (see  \cite[Theorem 13]{PR14}).
Then, we have 
\begin{equation*}
\begin{aligned}
 \|\mu_n- (\mu^+-\mu^-) \|^{a,b} \leq
&  \|\mu_n- \mu_{\varphi(n)} \|^{a,b} + 
\|\mu_{\varphi(n)}- (\mu^+-\mu^-) \|^{a,b} \\
& \leq \|\mu_n - \mu_{\varphi(n)} \|^{a,b} + 
W_1^{a,b}(\mu_{\varphi(n)}^++ \mu^- , \mu_{\varphi(n)}^-+ \mu^+) 
\\  & \leq  \|\mu_n- \mu_{\varphi(n)}\|^{a,b} +   W_1^{a,b}(\mu_{\varphi(n)}^+ , \mu^+)+
  W_1^{a,b}(\mu_{\varphi(n)}^- , \mu^-)
 \underset{n\to\infty}{\to} 0
\end{aligned}
\end{equation*}
since $(\mu_n)_n$ is a Cauchy sequence.
\end{proof}

\rev{We end this section with a characterization of the convergence for the norm. If a sequence $\mu_n$ of signed measures 
converges toward $\mu\in\M^s(\R^d)$, then for 
any decomposition of $\mu_n$ into two positive measures $\mu_n=\mu_n^+-\mu_n^-$ (not necessarily the Jordan decomposition), we have that each $\mu_n^+,\mu_n^-$ is the sum of two positive measures: 
$m_n^+,z_n^+$ and $m_n^-,z_n^-$, respectively. The measures $m_n^+$ and $m_n^-$ are the parts that converge respectively 
to $\mu^+$ and $\mu^-$. Both $m_n^+$ and $m_n^-$ are uniformly bounded and tight.
The measures $z_n^+$ and $z_n^-$ are  the residual terms that may be unbounded and not tight.
They compensate each other in the sense that $W_1^{a,b}(z_n^+,z_n^-)$ 
vanishes for large $n$.}
 \begin{theorem}\label{thm:decomposition}
 The two following statements are equivalent:

  \begin{equation*}
 (i) \quad \|\mu_n - \mu\|^{a,b} \underset{n \to \infty}{\to} 0.
 \end{equation*}

(ii) There exists \rev{four positive measures} $z_n^+, \; z_n^-, \; m_n^+, \; m_n^- \in \M(\R^d)$ such that
\begin{equation*}
 \begin{aligned}
  \mu_n^+ &= z_n^+ +m_n^+ ,\\
  \mu_n^-&=z_n^- +m_n^-,
 \end{aligned}
 \qquad \text{with }\qquad 
 \begin{aligned}
&  W_1^{a,b}(z_n^+,z_n^-) \underset{n \to \infty}{\to} 0,\\
 &   W_1^{a,b}(m_n^+,\mu^+) \underset{n \to \infty}{\to} 0,\\
  &    W_1^{a,b}(m_n^-,\mu^-) \underset{n \to \infty}{\to} 0,\\
 &\{m_n^+\}_n \text{ and }\{m_n^-\}_n  \text{ are tight and bounded in mass, }
 \end{aligned}
  \end{equation*}
  where $\mu= \mu^+ - \mu^-$ is the Jordan decomposition, and 
  $\mu_n= \mu_n^+ - \mu_n^-$ is any decomposition.
 \end{theorem}

 \begin{proof} We start by proving $(i) \Rightarrow (ii)$.
  We have $ \|\mu_n - \mu\|^{a,b}= \W_1^{a,b}(\mu_n,\mu)= W_1^{a,b}(\mu_n^+ + \mu^-, \mu_n^-+ \mu^+)=\underset{n \to \infty}{\to} 0$. 
  We denote by $a_n \leq (\mu_n^+ + \mu^-)$ and $b_n\leq(\mu_n^-+\mu^+)$ a choice of minimizers realizing $W_1^{a,b}(\mu_n^+ + \mu^-, \mu_n^-+ \mu^+)$, as well as $\pi_n$ being a minimizing transference plan from $a_n$ to $b_n$.
 \rev{We have $|a_n|=|b_n|$. Following the wording of Example \ref{ex}, the measures $a_n$ and $b_n$ are the transported mass, 
 and the measures 
 $(\mu_n^+-\mu^--a_n$ and $(\mu_n^--\mu^+-b_n)$ are the cancelled mass.}
 
 {\bf Step 1. The cancelled mass.}
 We define by $a_n^+$ and $b_n^-$ the largest transported mass
 which is respectively below $\mu_n^+$ and $\mu_n^-$
   \begin{equation*}
 \begin{aligned}
  a_n^+ &= \mu_n^+ \wedge a_n,\\
  a_n^-&= a_n-a_n^+,
 \end{aligned}
  \qquad
\begin{aligned}
 b_n^-&= \mu_n^- \wedge b_n, \\
 b_n^+ &= b_n-b_n^-.
\end{aligned}
\end{equation*}
The mass which is cancelled is then $r_n= r_n^++r_n^-: = (\mu_n^+ - a_n^+) + (\mu^- - a_n^-)$
 and $r_n^* = r_n^{*,-}+ r_n^{*,+} := (\mu_n^- - b_n^-) +( \mu^+ - b_n^+)$.
 The cancelled mass $r_n$ and $r_n^*$ are expressed here 
 as the sum of two positive measures. Indeed, it is clear by definition that
 $a_n^+ \leq \mu_n^+$, and
 since $a_n \leq \mu_n^+ + \mu^-$, Lemma \ref{lemma:prop_min_measures}, gives that 
 $a_n^-= a_n-a_n\wedge \mu_n^+ \leq \mu^-$. We reason the same way for $r_n^*$.
 Then, we have $W_1^{a,b}(\mu_n^+ + \mu^-, \mu_n^-+ \mu^+)= a\left( |\mu_n^+-a_n^+|
 +  |\mu^--a_n^-| +  |\mu_n^--b_n^-|+ |\mu^+-b_n^+|\right) +  b W_1(a_n,b_n)$. 
 Since $W_1^{a,b}(\mu_n^+ + \mu^-, \mu_n^-+ \mu^+)$ goes to zero, each of the five terms of the above decomposition goes to zero, and in particular, 
 $|\mu_n^+-a_n^+|\underset{n \to \infty}{\to} 0$ and $|\mu_n^--b_n^-|\underset{n \to \infty}{\to} 0$  which implies that
 \begin{equation}
 \label{s2}
  W_1^{a,b}(\mu_n^+-a_n^+,0)\underset{n \to \infty}{\to} 0, \qquad 
    W_1^{a,b}(\mu_n^--b_n^-,0)\underset{n \to \infty}{\to} 0.
 \end{equation}

 {\bf Step 2. The transported mass.} 
  The mass $a_n^+$ is split into two pieces: $\nu_n$ is sent to $\mu_n^-$, and $\xi_n$ is sent to $\mu^+$.
Denote by $\bar{a}_n^+$ the image
of $a_n^+$ under $\pi_n$ \MT{(using Definition \ref{def:image})},
then we define $\nu_n^* = \bar{a}_n^+ \wedge \mu_n^-$. \MT{(Still using Definition \ref{def:image})}, we denote by 
$\nu_n$ the image of $\nu_n^*$ under $\pi_n$.
Then, we define $\xi_n$ such that $a_n^+= \nu_n+ \xi_n,$ and we denote by 
$\xi_n^*$ the image of $\xi_n$ under $\pi_n$.
  By definition, we have 
  \begin{equation}
  \label{s1}
   W_1(a_n,b_n)=  W_1(\nu_n,\nu_n^*)+  W_1(\xi_n,\xi_n^*) + W_1(w_n,w_n^*)+ W_1(\a_n,\a_n^*),
  \end{equation}
  with $a_n^+ = \nu_n+ \xi_n,$ $w_n^*$ is defined so that $b_n^-= \nu_n^* + w_n^*,$ $w_n$ is the image of $w_n^*$ under $\pi_n$, $\a_n$ is defined so that $ \mu^-= w_n+ \a_n,$ 
  $\a_n^*$ is the image of $\a_n$ under $\pi_n$, and it can be checked that
  $\mu^+ = \xi_n^*+ \a_n^*$. 
  Since $ W_1(a_n,b_n)\underset{n \to \infty}{\to} 0$, each of the four term of the sum \eqref{s1} is going to zero. 
 
 {\bf Step 3. Conclusion.}
 
  Let us write 
  \begin{equation*}
   z_n^+ = \nu_n + (\mu_n^+-a_n^+),\quad 
    z_n^- = \nu_n^* + (\mu_n^--b_n^-), \quad 
    m_n^+ = \xi_n, \quad m_n^-= w_n^*.
  \end{equation*}
We show here that the sequences defined hereinabove satisfy the 
conditions stated in (ii).
First, we have $z_n^+ + m_n^+ = \nu_n^+ +  (\mu_n^+-a_n^+)+ \xi_n = \mu_n^+$ and
similarly, $z_n^- + m_n^-=\nu_n^* + (\mu_n^--b_n^-) + w_n^* = \mu_n^-.$

Then, we have
\begin{equation*}
\begin{aligned}
 W_1^{a,b}(z_n^+,z_n^-)&=  W_1^{a,b}(\nu_n + (\mu_n^+-a_n^+),\nu_n^* + (\mu_n^--b_n^-))\\
& \leq W_1^{a,b}(\nu_n ,\nu_n^* ) + W_1^{a,b}(\mu_n^+-a_n^+,\mu_n^--b_n^-) \quad \text{using Lemma \ref{lemma:prop_gen_dist}}\\
& \leq b W_1(\nu_n ,\nu_n^* ) +W_1^{a,b}(\mu_n^+-a_n^+,0)+ W_1^{a,b}(0,\mu_n^--b_n^-)
\\&\underset{n \to \infty}{\to} 0, \text{ using \eqref{s2} and \eqref{s1}.}\\ \end{aligned}
 \end{equation*}
Here, we also used that for $|\mu|= |\nu|, \; W_1^{a,b}(\mu,\nu) \leq  b W_1(\mu,\nu)$.
This is trivial with the definition of $W_1^{a,b}$. 
Now, we also have
\begin{equation*}
\begin{aligned}
  W_1^{a,b}(m_n^+, \mu^+) = W_1^{a,b}(\xi_n,\mu^+) 
 & \leq W_1^{a,b}(\xi_n,\xi_n^*) +  W_1^{a,b}(\xi_n^*,b_n^+)+  W_1^{a,b}(b_n^+,\mu^+) \quad (\text{triangular inequality})
\\ & =  W_1^{a,b}(\xi_n,\xi_n^*) + W_1^{a,b}(\a_n^*,0) + W_1^{a,b}(\mu^+-b_n^+,0)
\end{aligned}
\end{equation*}
since $ \a_n^*+ \xi_n^*=b_n^+$. 
  We know that $W_1^{a,b}(\xi_n,\xi_n^*)\leq  b W_1(\xi_n,\xi_n^*)\underset{n \to \infty}{\to} 0$ using \eqref{s1}, 
  and that  $W_1^{a,b}(\mu^+-b_n^+,0)\underset{n \to \infty}{\to} 0$ using \eqref{s2}.
  Let us explain now why $W_1^{a,b}(\a_n^*,0) \underset{n \to \infty}{\to} 0$.
We remind that $W_1(\a_n,\a_n^*)\underset{n \to \infty}{\to} 0$, $\a_n \leq a_n^- \leq \mu^-, \; \a_n^*\leq b_n^+ \leq \mu^+$.
Since $(\a_n)_n$ is uniformly bounded in mass, then there exists $\a \in \M(\R^d)$ such that
$\a_{\varphi(n)} \underset{n\to \infty}{\rightharpoonup} \alpha$ vaguely
(see Lemma \ref{lemma:weak_compactness}).
We have also that $(\a_{\varphi(n)})_n$ is tight, since $\a_{\varphi(n)} \leq \mu^-$ which has a finite mass. 
Using Theorem 13 of \cite{PR12}, we deduce that 
$W_1^{a,b} (\a_{\varphi(n)}, \a)\underset{n \to \infty}{\to} 0 $.
Then, $W_1^{a,b} (\a^*_{\varphi(n)}, \a) \leq W_1^{a,b} (\a^*_{\varphi(n)}, \a_{\varphi(n)}) + 
 W_1^{a,b}(\a_{\varphi(n)}, \a) \leq 
 W_1 (\a^*_{\varphi(n)}, \a_{\varphi(n)}) + 
  W_1^{a,b}(\a_{\varphi(n)}, \a)
  \underset{n \to \infty}{\to} 0$.
 Then, using again Theorem 13 of \cite{PR12}, we deduce that 
 $\a^*_{\varphi(n)} \underset{n\to \infty}{\rightharpoonup} \alpha$ vaguely.
Since $\a_n \leq \mu^-$, we have $\a \leq \mu^-$. Likewise, 
$\a^*_n \leq \mu^+$ implies $\a \leq \mu^+$.
Since
$\mu^- \perp \mu^+$, we have $\a=0$.
We have $W_1^{a,b}(\a_{\varphi(n)}, 0) \underset{n \to \infty}{\to} 0$
and $W_1^{a,b}(\a_{\varphi(n)}, 0) \underset{n \to \infty}{\to} 0$.
The sequence $(\a_n)_n$ satisfies the following property: each of its subsequences admits a subsequence converging to zero.
Thus, we have that the whole sequence is converging to zero, i.e. 
$W_1^{a,b}(\a_n, 0) \underset{n \to \infty}{\to} 0$
and $W_1^{a,b}(\a^*_n, 0) \underset{n \to \infty}{\to} 0$.
Lastly, the tightness of $(m_n^+)_n$ and $(m_n^-)_n$ is given 
again by Theorem 13 of \cite{PR12}, since  $W_1^{a,b}(m_n^{\pm} , \mu^{\pm})\underset{n \to \infty}{\to} 0$.\\
  
  We prove now that $(ii) \Rightarrow (i)$.
  Let us assume (ii). We have 
  \begin{eqnarray*}
   \|\mu_n-\mu\|^{a,b} &=& W_1^{a,b}(\mu_n^+ +\mu^-,\mu_n^-+\mu^+) =   W_1^{a,b}(z_n^+ +m_n^+ +\mu^-,z_n^-+ m_n^-+\mu^+) \\
   &\leq& W_1^{a,b} (z_n^+,z_n^-)+ W_1^{a,b} (m_n^+,\mu^+)+ W_1^{a,b} (\mu^-,m_n^-) \underset{n \to \infty}{\to} 0,
  \end{eqnarray*}
\rev{where the last inequality comes from Lemma \ref{lemma:prop_gen_dist}.} This proves which is (i).
  \end{proof}
  
\section{Application to the transport equation with source term}
\label{sec:transport_equation}

\comMagali{This section is devoted to the use of the norm defined in Definition \ref{def:norm} to guarantee existence, 
uniqueness, and stability with respect to initial condition for the transport equation \eqref{transport_equation}.
}
\rev{
\begin{definition}[Measure-valued weak solution]
\label{def:mv_solution}
A measure-valued weak solution to \eqref{transport_equation} is a map 
$\mu \in \mathcal{C}^0([0,1]; \M^s(\R^d))$ such that $\mu_{t=0}=\mu_0$ and  for 
 all 
 $\Phi \in \D(\R^d)$
 it holds
\begin{equation}\label{eq:goal}
\dfrac{d}{dt}\langle \Phi ,  \mu_t \rangle  =\langle v[\mu_t]. \nabla \Phi ,\mu_t\rangle
+ \langle h[\mu_t], \Phi\rangle,
\end{equation}
where
 $$\langle \mu_t,\Phi\rangle:= \dst\int_{\R^d} \Phi(x) d\mu_t(x).$$
\end{definition}
}

\subsection{Estimates of the norm under flow action}
\label{subsection:41}

In this section, we extend the action of flows on probability measures to signed measures,
and state some estimates about the variation
of $\|\mu-\nu\|^{a,b}$ after the action of a flow on $\mu$ and $\nu$. 
Notice that for $\mu\in \M^s(\R^d)$ and $T$ a map, we have $T \# \mu = T \# \mu^+ - T\# \mu^-$, 
 where $\mu= \mu^+ -\mu^-$ is any decomposition of $\mu$. 
 Observe that in general, given $\mu\in \M^s(\R^d)$ and $T:\R^d\mapsto \R^d$ a Borel map, 
 it only holds $|T\#\mu|\leq |\mu|$, even by choosing the Jordan decomposition for $(\mu^+,\mu^-)$, since it may hold 
 that $T\#\mu^+$ and $T\#\mu^-$ are not orthogonal.
 However, if $T$ is injective (as it will be in the rest of the paper), it holds 
 $T\#\mu^+\perp T\#\mu^-$, hence $|T\#\mu|= |\mu|$.
\begin{lemma}
\label{lemma:characteristics}
 For $v(t,x)$ measurable in time, uniformly Lipschitz in space, and uniformly bounded, we denote by 
 $\Phi_t^v$ the flow it generates, i.e. the unique solution to
 \begin{equation*}
  \dfrac{d}{dt}\Phi_t^v= v(t,\Phi_t^v), \qquad \Phi_0^v=I_d.
 \end{equation*}
 Given $\mu_0 \in \M^s(\R^d)$, 
then, 
$\mu_t= \Phi_t^v \# \mu_0$ is the unique solution of the linear transport equation
\begin{equation*}
\left\{
 \begin{aligned}
&\dfrac{\p}{\p t}\mu_t  + \nabla.(v(t,x)\mu_t)=0,\\
 &\mu_{|t=0}= \mu_0
 \end{aligned}
\right.
\end{equation*}
in $\C([0,T], \M^s(\R^d))$.
 \end{lemma}

 \begin{proof}
  The proof is a direct consequence of \cite[Theorem 5.34]{Villani_topics} combined 
  with  \cite[Theorem 2.1.1]{BressanPiccoli_book}. 
 \end{proof}

\begin{lemma}\label{lemma:flow}
Let $v$ and $w$ be two vector fields, both satisfying 
for all $t\in [0,1]$ and $x,y\in\R^d,\;$ the following properties: $$|v(t,x)-v(t,y)|\leq L|x-y|,\qquad
|v(t,x)|\leq M.$$
Let $\mu$ and $\nu$ be two  measures of $\M^s(\R^d)$. Then

\noindent
 \begin{itemize}
  \item $\|\phi_t^v \# \mu -\phi_t^v \# \nu \|^{a,b} \leq e^{Lt} \|\mu-\nu\|^{a,b}$  
  \item $\| \mu -\phi_t^v \# \mu \|^{a,b} \leq b \; t \; M|\mu|,$
 \item $\|\phi_t^v \# \mu -\phi_t^w \# \mu \|^{a,b} \leq  b|\mu|\frac{(e^{Lt}-1)}{L} \|v-w\|_{\rev{L^{\infty}(0,1;\;\C^0)}}$
 \item $\|\phi_t^v \# \mu -\phi_t^w \# \nu \|^{a,b} \leq e^{Lt} \|\mu-\nu\|^{a,b} + b \; \MT{\min\{|\mu|,|\nu|\}} \frac{(e^{Lt}-1)}{L} \|v-w\|_{\rev{L^{\infty}(0,1;\;\C^0)}}$
 \end{itemize}
\end{lemma}
\begin{proof}
The first three inequalities follow from \cite[Proposition 10]{PR16}.
For the first inequality, we write
\begin{equation*}
\begin{aligned}
 \|\phi_t^v \# \mu -\phi_t^v \# \nu \|^{a,b} =  \W_1^{a,b}(\phi_t^v \# \mu,\phi_t^v \# \nu) 
 &= \W_1^{a,b}(\phi_t^v \# \mu^+ - \phi_t^v \# \mu^-,\phi_t^v \# \nu^+ - \phi_t^v \# \nu^- )
  \\ &= W_1^{a,b}(\phi_t^v \# (\mu^+ + \nu^-),\phi_t^v \# (\mu^- + \nu^+) ) 
  \\& \leq e^{Lt} W_1^{a,b}(\mu^+ + \nu^-, \mu^-+ \nu^+)\quad \text{ by \cite[Prop. 10]{PR16}}
  \\& = e^{Lt} \|\mu-\nu\|^{a,b}.
\end{aligned}
\end{equation*}
For the second inequality, 
\begin{equation*}
 \begin{aligned}
  \W_1^{a,b}(\mu, \phi_t^v \#\mu) &= W_1^{a,b}(\mu^+ + \phi_t^v \# \mu^-, \mu^- + \phi_t^v \# \mu^+)\\
  &\leq  W_1^{a,b}(\mu^+,   \phi_t^v \# \mu^+)+  W_1^{a,b}(\mu^-,   \phi_t^v \# \mu^-) \quad \text{(Lemma \ref{lemma:prop_gen_dist})}
 \\ & \leq b\; t \; \|v\|_{\C^0}(|\mu^+| + |\mu^-|) \quad \text{ by \cite[Prop. 10]{PR16}}
 \\& = b\; t \;  \|v\|_{L^{\infty}(0,1;\;\C^0(\R))}|\mu| \quad \text{since $\mu= \mu^+-\mu^-$ is the Jordan decomposition.} 
 \end{aligned}
\end{equation*}
The third inequality is given by
\begin{equation*}
\begin{aligned}
 \|\phi_t^v \#  \mu- \phi_t^w \# \mu\|^{a,b}&= \W_1^{a,b}(\phi_t^v \# \mu^++ \phi_t^w \# \mu^-,\phi_t^w \# \mu^++ \phi_t^v \# \mu^-) \\
 &\leq W_1^{a,b}(\phi_t^v \# \mu^+,\phi_t^w \# \mu^+) + W_1^{a,b}( \phi_t^w \# \mu^-, \phi_t^v \# \mu^-) \\
&\leq b W_1(\phi_t^v \# \mu^+,\phi_t^w \# \mu^+) + W_1( \phi_t^w \# \mu^-, \phi_t^v \# \mu^-) \\
&\leq (| \mu^+| + |\mu^-| )\dfrac{(e^{Lt}-1)}{L}\|v-w\||_{L^{\infty}(0,1;\;\C^0(\R))} \quad \text{using  \cite[Prop. 10]{PR16} with $\mu=\nu$.}
\end{aligned}
\end{equation*}

 The last inequality is deduced from the first and the third one using triangular inequality. 
\end{proof}

\subsection{A scheme for computing solutions of the transport equation}
\label{s-scheme}

In this section, we build a solution to \eqref{transport_equation} as the limit of a sequence of approximated solutions defined in the following scheme. We then prove that \eqref{transport_equation} admits a unique solution.

Consider $\mu_0 \in \M^s(\R^d)$ such that $\supp(\mu_0)\subset \mathcal{K}$,  with $\mathcal{K}$ compact. Let $v \in \C^{0,Lip}(\M^s(\R^d), \C^{0,Lip}(\R^d))$ and $h \in \C^{0,Lip}(\M^s(\R^d),\M^s(\R^d) )$ satisfying 
\ref{H1}-\ref{H2}-\ref{H3}.
 We now define a sequence $(\mu_t^k)_k$  of approximated solutions for \eqref{transport_equation}
through the following Euler-explicit-type iteration scheme.
For simplicity of notations, we define a solution on the time interval $[0,1]$.
\newline

 \noindent\fbox{\parbox{\linewidth-2\fboxrule-2\fboxsep}{
\begin{center}
 \textsc{Scheme}
\end{center}

\noindent{\bf Initialization.}
Fix $k\in \N$. Define $\Delta t= \dfrac{1}{2^k}$.
Set  $\mu_0^k = \mu_0$.

\noindent{\bf Induction.}
Given $\mu_{i\Delta t}$ for $i\in \{0,1, \dots,2^k-1\}$,
define $v^k_{i\Delta t}:=v[\mu_{i\Delta t}^k]$ 
and 
\begin{equation}
\label{eqt:scheme}
\mu_t^k= \Phi_{t-i\Delta t}^{v_{i\Delta t}} \# \mu_{i\Delta t}^k+ (t-i\Delta t)h[\mu_{i\Delta t}^k], \qquad t\in [i{\Delta t}, (i+1)\Delta t].
    \end{equation}
}}
\\

\rev{The scheme is a natural operator splitting: the flow $ \Phi_{t-i\Delta t}$ encodes the transport part $\partial_t\mu+\div(v\mu)=0$ while $(t-i\Delta t)h$ encodes the reaction $\partial_t\mu = h$.}
\begin{proposition}\label{prop:Cauchy_sequence}
 The sequence $(\mu_t^k)_k$ defined in the scheme above is a Cauchy sequence in the space
 $\C^0([0,1], \M^s(\R^d), \|.\|)$ with 
 \begin{equation*}
 \|\mu_t\|= \sup_{t\in [0,1]} \|\mu_t\|^{a,b}.  
 \end{equation*}
 Moreover, it is uniformly bounded in mass \rev{and compactly supported}, i.e. 
 \begin{equation}
 \label{uniform_estimate}
|\mu^k_t| <P \rev{t} +|\mu_0|, \qquad \rev{ 
supp\{\mu_t\} \subset B(0,R'+tM),\; 
 \qquad 0 \leq t\leq 1,
}
 \end{equation}
for  $R'$ such that 
$\left(\supp\{\mu_0\}\cup B_0(R)\right)\subset B_0(R')$.
\end{proposition}

\rev{Let us mention that the estimate \eqref{uniform_estimate} is expected at the discrete level from the PDE \eqref{transport_equation} with the assumptions \ref{H1}, \ref{H2}, \ref{H3}.
Indeed, the transport part should preserve mass (more precisely $|T\#\mu|\leq|\mu|$
as discussed in subsection\ref{subsection:41}, while the reaction term $|h|\leq P$ gives a mass growth that is at most linear $Pt$.
Likewise, the support estimate is expected from the PDE since $h$ has support in $B_0(R)$ (no mass created out of this ball) and transport cannot expand the support faster than $|v|\leq M$.}

\begin{proof} Let $L$ be the Lipschitz constant in \ref{H2}.
We assume to have $k$ sufficiently large to have $e^{Lt}\leq 1+2Lt$ for all $t\leq[0,\Delta t]$. This holds e.g. for $L\Delta t\leq 1$, hence $k\geq \log_2(L)$.

 We also notice that the sequence built by the scheme satisfies 
 \begin{equation}\label{bounded_mass}
  |\mu_t^k|\leq P \rev{t} + |\mu_0|, \quad t\in[0,1],
 \end{equation}
 where $P$ is such that $\left|h[\mu]\right| \leq P$ by \ref{H3}.
Indeed, it holds for $t\in [i{\Delta t}, (i+1)\Delta t]$
\begin{equation*}
   |\mu_t^k|\leq |\Phi_t^{v_{i\Delta t}} \# \mu_{i\Delta t}^k|+ \Delta t|h[\mu_{i\Delta t}^k]|
\leq |\mu_{i\Delta t}^k| + \Delta t P,
\end{equation*}
 and the result follows by induction \rev{on $i$ (for $k$ fixed)}.
  This proves \eqref{uniform_estimate}.
 The sequence $(\mu_t^k)_{k\in\N}$ also has uniformly bounded support. Indeed, \rev{first observe that $\supp\{\mu\} = \supp\{ \mu^+\} \cup \supp\{\mu^-\}$, where $(\mu^+, \mu^-)$ is the Jordan decomposition of $\mu$. Choose $\mathcal{K}$ such that $\supp\{\mu_0\}\subset \mathcal{K}$ and } use \eqref{eqt:scheme} and \ref{H2}-\ref{H3} to write
 \begin{equation*}
  \supp\{\mu_t^k\} \subseteq\mathcal{K}_{t,M, R} , 
 \end{equation*}
 with
 \begin{equation*}
\mathcal{K}_{t,M,R}:= \{x \in \R^d,\; x=x_{\mathcal{K},R}+x', \; x_{\mathcal{K},R}\in \mathcal{K}\cup B_0(R),\; \|x'\|\leq tM  \}.
\end{equation*}
Take now $R'$ such that $\mathcal{K}\cup B_0(R)\subset B_0(R')$. Then, it holds $\mathcal{K}_{t,M,R}\subset B(0,R'+\rev{t}M)$. Since such set does not depend on $t$ \rev{for $0\leq t\leq 1$}, while $M,R$ are fixed, then $\mu^k_t$ have uniformly bounded support.

 We now follow the notations of \cite{PR12}
and define $m_j^k:=\mu_{\frac{j}{2^k}}^k$, $v_j^k:= v[m_j^k]$
and the corresponding flow $f_t^{j,k}: = \phi_t^{v_j^k}$.
Fix $k\in \N$ and $t\in [0,1]$. Define $j \in \{0,1,\dots, 2^k\}$ such that $t \in\left] \frac{j}{2^k},\frac{j+1}{2^{k}}\right]$
\rev{The following inequalities rely on Lemma \ref{lemma:flow} and 
\ref{H1}, \ref{H2}, \ref{H3}.}

\noindent{\bf First case.}
If $t \in \left] \frac{j}{2^k},\frac{2j+1}{2^{k+1}}\right]$.

\begin{center}
\end{center}
We call $t'=t-\frac{j}{2^k} \leq\frac{1}{2^{k+1}} $ and we obtain
\begin{equation*}
\begin{aligned}
\W_1^{a,b}( \mu_t^k, \mu_t^{k+1})&=\W_1^{a,b}(f_{t'}^{j,k} \#m_j^k+t'h[m_j^k],f_{t'}^{2j,k+1} \#m_{2j}^{k+1}+t'h[m_{2j}^{k+1}]) \\
& \leq \W_1^{a,b}(f_{t'}^{j,k} \#m_j^k,f_{t'}^{2j,k+1} \#m_{2j}^{k+1}) + \W_1^{a,b}(t'h[m_j^k], t'h[m_{2j}^{k+1}])  \\
&\leq  e^{Lt'}\W_1^{a,b}(m_j^k,m_{2j}^{k+1}) + |m_j^k|\dfrac{(e^{Lt'}-1)}{L}\|v_j^k-v_{2j}^{k+1}\|_{\C^0(\R^d)}+ t'Q\W_1^{a,b}(m_j^k,m_{2j}^{k+1})\\
&\leq \W_1^{a,b}(m_j^k,m_{2j}^{k+1})\left(e^{Lt'} + (P+|\mu^0|)\dfrac{1}{L}(e^{Lt'}-1)+t'Q \right)\\
\end{aligned}
\end{equation*}
Since it holds
\begin{equation*}
 e^{Lt'} \leq 1+ 2 Lt' \leq  1+ 2L 2^{-(k+1)}, \quad \dfrac{(e^{Lt'}-1)}{L} \leq 2\cdot 2^{-(k+1)},
\end{equation*}

we have 
\begin{equation}
\label{fc}
\begin{aligned}
\|\mu_t^k-\mu_t^{k+1}\|^{a,b}& \leq \|m_j^k-m_{2j}^{k+1}\|^{a,b}\left(1+2^{-(k+1)}\left(2L+ 2(P+|\mu^0|) + Q  \right) \right), \quad t \in \left[ \frac{j}{2^k},\frac{2j+1}{2^{k+1}}\right].
\end{aligned}
\end{equation}

\noindent{\bf Second case.}
If $t \in \left] \frac{2j+1}{2^{k+1}},\frac{j+1}{2^{k}}\right]$.

\begin{center}
\includegraphics[width=1\textwidth]{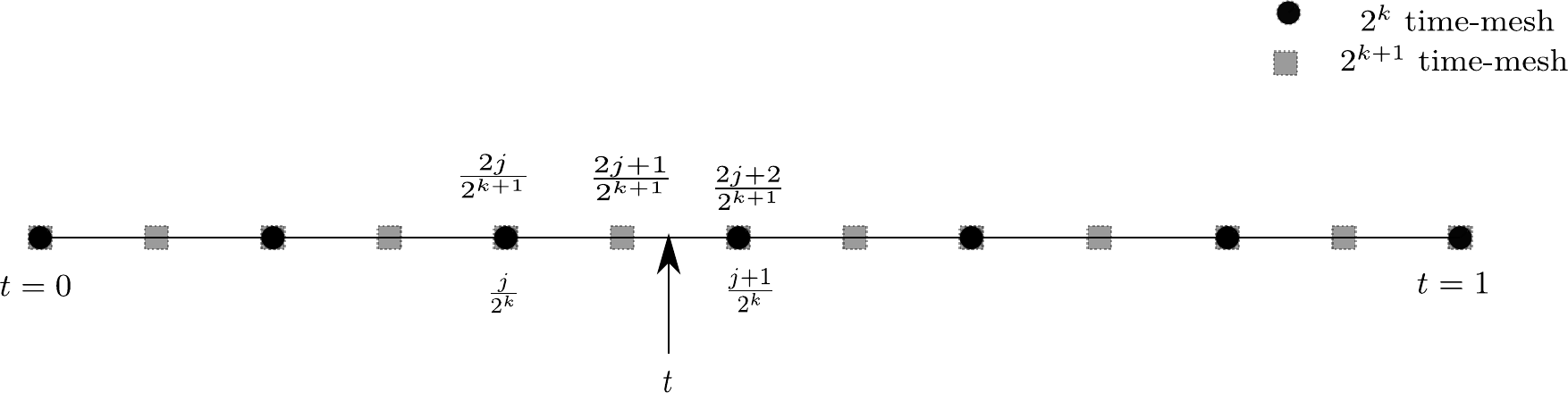}
\end{center}
We call $t'=t-\frac{2j+1}{2^{k+1}} \leq\frac{1}{2^{k+1}} $ and
we obtain
\begin{equation*}
\begin{aligned}
 &\mu_t^k = f_{t'+ \frac{1}{2^{k+1}}}^{j,k}\# m_j^k+ \left(t'+\frac{1}{2^{k+1}}\right)h[m_j^k]= 
f_{t'}^{j,k}\#  f_{\frac{1}{2^{k+1}}}^{j,k} \# m_j^k+ t'h[m_j^k] +\frac{1}{2^{k+1}}h[m_j^k],\\
 &\mu_t^{k+1} = f_{t'}^{2j+1,k+1} \# \left(f_{\frac{1}{2^{k+1}}}^{2j,k+1} \# m_{2j}^{k+1}+ \dfrac{1}{2^{k+1}}h[m_{2j}^{k+1}]\right)+t'h\left[ f_{\frac{1}{2^{k+1}}}^{2j,k+1} \# m_{2j}^{k+1}+ \dfrac{1}{2^{k+1}}h[m_{2j}^{k+1}]\right]
\\&  =  f_{t'}^{2j+1,k+1} \# f_{\frac{1}{2^{k+1}}}^{2j,k+1} \# m_{2j}^{k+1} + \dfrac{1}{2^{k+1}} f_{t'}^{2j+1,k+1} \# h[m_{2j}^{k+1}] +t'h\left[ f_{\frac{1}{2^{k+1}}}^{2j,k+1} \# m_{2j}^{k+1}+ \dfrac{1}{2^{k+1}}h[m_{2j}^{k+1}]\right].
\end{aligned}
\end{equation*}
It then holds
\begin{equation}\label{e0}\begin{aligned}
\|\mu_t^k- \mu_t^{k+1}\|^{a,b} &\leq \W_1^{a,b}\left(f_{t'}^{j,k}\#  f_{\frac{1}{2^{k+1}}}^{j,k} \# m_j^k,f_{t'}^{2j+1,k+1} \# f_{\frac{1}{2^{k+1}}}^{2j,k+1} \# m_{2j}^{k+1}\right)\\
&+\dfrac{1}{2^{k+1}}\W_1^{a,b}\left(h[m_j^k],f_{t'}^{2j+1,k+1} \# h[m_{2j}^{k+1}] \right)
\\& +t'\W_1^{a,b}\left(h[m_j^k],h\left[ f_{\frac{1}{2^{k+1}}}^{2j,k+1} \# m_{2j}^{k+1}+ \dfrac{1}{2^{k+1}}h[m_{2j}^{k+1}] \right] \right).
\end{aligned}
\end{equation}
Use now Lemma \ref{lemma:flow} to prove the following estimate
\begin{equation*}
 \begin{aligned}
  \W_1^{a,b}&\left(f_{t'}^{j,k}\#  f_{\frac{1}{2^{k+1}}}^{j,k} \# m_j^k,f_{t'}^{2j+1,k+1} \# f_{\frac{1}{2^{k+1}}}^{2j,k+1} \# m_{2j}^{k+1}\right)
\\ & \leq (1+2L 2^{-(k+1)}) \W_1^{a,b}\left( f_{\frac{1}{2^{k+1}}}^{j,k} \# m_j^k, f_{\frac{1}{2^{k+1}}}^{2j,k+1} \# m_{2j}^{k+1}\right)
+2^{-(k+1)}2P\|v_j^k-v_{2j+1}^{k+1}\|_{\C^0(\R^d)}. \\
\end{aligned}
\end{equation*}
Since, according to the first case, it holds both
\begin{equation*}
 \W_1^{a,b}\left( f_{\frac{1}{2^{k+1}}}^{j,k} \# m_j^k, f_{\frac{1}{2^{k+1}}}^{2j,k+1} \# m_{2j}^{k+1}\right)
\leq \|m_j^k-m_{2j}^{k+1}\|^{a,b}\left(1+2^{-(k+1)}\left(2L+ 2(P+|\mu^0|)  \right) \right)
\end{equation*}
and
\begin{equation*}
\begin{aligned}
   \|v_j^k-v_{2j+1}^{k+1}\|_{\C^0(\R^d)}  \leq K \W_1^{a,b}(m_j^k,m_{2j+1}^{k+1})& \leq K \W_1^{a,b}(m_j^k,m_{2j}^{k+1})+K \W_1^{a,b}(m_{2j}^{k+1},m_{2j+1}^{k+1})\\
   & \leq K \W_1^{a,b}(m_j^k,m_{2j}^{k+1})+K \W_1^{a,b}(m_{2j}^{k+1},m_{2j+1}^{k+1})\\
  &=K\W_1^{a,b}(m_j^k,m_{2j}^{k+1})+K \W_1^{a,b}(m_{2j}^{k+1},f_{\frac{1}{2^{k+1}}}^{2j,k+1} \# m_{2j}^{k+1})\\
&=K\W_1^{a,b}(m_j^k,m_{2j}^{k+1})+KM2^{-(k+1)},\\
 \end{aligned}
 \end{equation*}
we have
\begin{equation}\label{e1}
\begin{aligned}
 \W_1^{a,b}&\left(f_{t'}^{j,k}\#  f_{\frac{1}{2^{k+1}}}^{j,k} \# m_j^k,f_{t'}^{2j+1,k+1} \# f_{\frac{1}{2^{k+1}}}^{2j,k+1} \# m_{2j}^{k+1}\right)
\\ &
\leq \|m_j^k-m_{2j}^{k+1}\|^{a,b}\left(1+2^{-(k+1)}\left(4L+ 2(P+|\mu^0|)(1+L) +2KP  \right) \right)
+2^{-2(k+1)}2PKM.
\end{aligned}
\end{equation}

Moreover, it also holds both
\begin{equation}\label{e2}
\begin{aligned}
 \W_1^{a,b}&\left(h[m_j^k],f_{t'}^{2j+1,k+1} \# h[m_{2j}^{k+1}] \right) \\
 &\leq
  \W_1^{a,b}\left(h[m_j^k],f_{t'}^{2j+1,k+1} \# h[m_{j}^{k}] \right)
  +  \W_1^{a,b}\left(f_{t'}^{2j+1,k+1} \# h[m_{j}^{k}] ,f_{t'}^{2j+1,k+1} \# h[m_{2j}^{k+1}] \right)
 \\ &\leq t'MP+e^{Lt'}Q\|m_j^k- m_{2j}^{k+1}\|^{a,b} \leq + MP2^{-(k+1)} + (1+ 2L2^{-(k+1)})\|m_j^k- m_{2j}^{k+1}\|^{a,b} ,
 \end{aligned}
\end{equation}
and 
\begin{equation}\label{e3}
\begin{aligned}
 \W_1^{a,b}&\left(m_j^k, f_{\frac{1}{2^{k+1}}}^{2j,k+1} \# m_{2j}^{k+1}+ \dfrac{1}{2^{k+1}}h[m_{2j}^{k+1}]  \right)
\\&\leq \W_1^{a,b}\left(m_j^k, f_{\frac{1}{2^{k+1}}}^{2j,k+1} \# m_{2j}^{k+1} \right)
 +2^{-(k+1)}\W_1^{a,b}\left(0, h[m_{2j}^{k+1}] \right)
 \\ & \leq  \W_1^{a,b}\left(m_j^k,m_{2j}^{k+1}\right)
+ \W_1^{a,b}\left(m_{2j}^{k+1}, f_{\frac{1}{2^{k+1}}}^{2j,k+1} \# m_{2j}^{k+1} \right)
+ 2^{-(k+1)}aP \\
&\leq \|m_{j}^{k}- m_{2j}^{k+1}\|^{a,b} + 2^{-(k+1)}(|\mu^0|+P(1+a) ). 
 \end{aligned}
 \end{equation}

Plugging \eqref{e1}, \eqref{e2} and \eqref{e3} into \eqref{e0}, and combining it with \eqref{fc} we find {\bf in both cases}
\begin{equation}
\label{both_cases}
\begin{aligned}
\|\mu_t^k- \mu_t^{k+1}\|^{a,b} \leq (1+2^{-k}C_1) \|m_j^k- m_{2j}^{k+1}\|^{a,b} +C_2 2^{-2k}, 
\quad t \in \left] \frac{j}{2^{k}},\frac{j+1}{2^{k}}\right],
\end{aligned}
\end{equation}
with 
\begin{equation*}
  C_1= \left(1+3L+ (P+|\mu^0|)(1+L) +KP +Q \right) , \quad C_2=\dfrac{1}{4}(MP(1+2K)+ |\mu^0|+P(1+a)).    
     \end{equation*}
     \rev{In particular, plugging $t=(j+1)/2^k$ in \eqref{both_cases}, we get
\begin{equation*}\begin{aligned}
\|m_{j+1}^k-m_{2(j+1)}^{k+1}\|^{a,b} \leq (1+2^{-k}C_1) \|m_j^k- m_{2j}^{k+1}\|^{a,b} +C_2 2^{-2k}, 
\end{aligned}
\end{equation*}     
and by induction on $j$ (for $k$ fixed), we obtain 
 \begin{equation*}
 \|m_{2^k}^k-m_{2^{k+1}}^{k+1}\|^{a,b} \leq \sum_{j=0}^{2^k-1} (1+2^{-k}C_1)^j 2^{-2k}C_2  \leq 
  \dfrac{C_2}{C_1}(e^{C_1}-1)2^{-k}.
 \end{equation*}
From \eqref{both_cases} it holds
 \begin{equation*}
  \|\mu_t^k- \mu_t^{k+1}\| \leq \|m_{2^k}^k-m_{2^{k+1}}^{k+1}\|^{a,b} ,
 \end{equation*}
 and then we conclude
  \begin{equation*}
  \|\mu_t^k- \mu_t^{k+1}\| \leq \dfrac{C_2}{C_1}(e^{C_1}-1)2^{-k}.
 \end{equation*}
 }
Since the right hand side is the term of a convergent series, then $(\mu^k_t)_k$ is a Cauchy sequence.
\end{proof}

\subsection{Proof of Theorem \ref{t-exun}}

In this section, we prove Theorem \ref{t-exun}, stating existence and uniqueness of the solution to the Cauchy problem associated to \eqref{transport_equation}. The proof is based on the proof of the same result for positive measures written in \cite{PR16}. 
We first focus on existence. 

{\bf Step 1. Existence.}
Observe that the sequence given by the scheme $(\mu_t^k)_k$ is a Cauchy sequence (Proposition \ref{prop:Cauchy_sequence}) in the space
$\left(\C^0[0,1], \;\M^s(\R^d)\right)$  is uniformly bounded in mass \rev{and tight (see Proposition \ref{prop:Cauchy_sequence})} . Then, by using Theorem \ref{thm:Banach}, we define
\begin{equation*}
 \mu_t:= \lim\limits_{k\to \infty} \mu_t^k, 
\end{equation*}
\rev{where the convergence holds in the space $\C^0\left([0,1], \M^s(\R^d)\right).$}
Denote the following \rev{for $\varphi \in \D((0,1)\times \R^d)$}: $$\langle \mu,\varphi\rangle:= \dst\int_{\R^d} \varphi(t,x) d\mu_t(x).$$
The goal is to prove that for all $\varphi \in \D((0,1) \times \R^d)$, we have
\begin{equation*}
  \dst\int_0^1 dt\left( \langle  \mu_t , \p_t \varphi (t,x) + v[\mu_t]. \nabla \varphi(t,x)\rangle + \langle h[\mu_t], \varphi(t,x) \rangle \right) =0.
\end{equation*}
\rev{This implies (it is equivalent) that for all $\phi \in \D((0,1) \times \R^d)$, \eqref{eq:goal} holds (see \cite[chapter 8]{AGS08}.}
We first notice that
\begin{equation*}
 \begin{aligned}
  \sum_{j=0}^{2^k-1}\dst\int_{j \Delta t}^{(j+1)\Delta t}
  dt \left( \langle  \mu_t^k , \p_t \varphi (t,x) + v[\mu_{j\Delta t}^k]. \nabla \varphi(t,x)\rangle + \langle h[\mu^k_{j\Delta t}], \varphi(t,x) \rangle \right)  \underset{k\to \infty}{\longrightarrow }0
 \end{aligned}
\end{equation*}
Indeed, $\nu_t:= \phi_t^v \# \nu_0$ is a weak solution of
$\frac{\p}{\p t} \nu_t + \nabla.\left(v(x) \nu_t\right)=0$ with $v$ a fixed vector field,
and $\eta_t= \eta_0 +t h$ is a weak solution of 
$\frac{\p}{\p t} \eta_t = h$, with $h$ a fixed measure.
\rev{We apply this to $\mu^k$ piecewise on each time interval.} 
It then holds, \rev{using that $\mu_t^k$ satisfies \eqref{eqt:scheme}}
\begin{equation*}
 \begin{aligned}
 \Big| \sum_{j=0}^{2^k-1}\dst\int_{j \Delta t}^{(j+1)\Delta t}&
  dt \left( \langle  \mu_t^k , \p_t \varphi (t,x) + v[\mu_{j\Delta t}^k]. \nabla \varphi(t,x)\rangle + \langle h[\mu^k_{j\Delta t}], \varphi(t,x) \rangle \right) \Big| \\
  &=\left| \sum_{j=0}^{2^k-1}\dst\int_{j \Delta t}^{(j+1)\Delta t}
  dt \; \langle  (t-j\Delta t) h[\mu_{j \Delta t}^k] ,  v[\mu_{j\Delta t}^k]. \nabla \varphi(t,x)\rangle  \right| 
 \\ &
 \leq M P \|\nabla \varphi\|_{\infty} 2^{-(k+1)}
  \underset{k\to \infty}{\longrightarrow }0.
 \end{aligned}
\end{equation*}

Now, to guarantee \eqref{eq:goal}, it is enough to prove that 
\begin{equation*}
\begin{aligned}
 \lim\limits_{k\to \infty} \Big|
 &\dst\int_0^1 dt\left( \langle  \mu_t , \p_t \varphi (t,x) + v[\mu_t]. \nabla \varphi(t,x)\rangle + \langle h[\mu_t], \varphi(t,x) \rangle \right) 
 \\&-\sum_{j=0}^{2^k-1}\dst\int_{j \Delta t}^{(j+1)\Delta t}
  dt \left( \langle  \mu_t^k , \p_t \varphi (t,x) + v[\mu_{j\Delta t}^k]. \nabla \varphi(t,x)\rangle + \langle h[\mu_t^k], \varphi(t,x) \rangle \right) 
 \Big|=0
 \end{aligned}
\end{equation*}
We have
\begin{equation*}
 \begin{aligned}
 \Big|
 \dst\int_0^1 dt\left( \langle  \mu_t , \p_t \varphi (t,x) \rangle\right) 
 -\sum_{j=0}^{2^k-1}\dst\int_{j \Delta t}^{(j+1)\Delta t}
  dt \left( \langle  \mu_t^k , \p_t \varphi (t,x)  \rangle \right) 
 \Big|
 \leq \|\p_t\varphi\|_{\infty} \|\mu_t- \mu_t^k\|  \underset{k\to \infty}{\longrightarrow }0,
 \end{aligned}
\end{equation*}

\begin{equation*}
\begin{aligned}
 \Big|
 &\dst\int_0^1 dt  \langle h[\mu_t], \varphi(t,x) \rangle  
-\sum_{j=0}^{2^k-1}\dst\int_{j \Delta t}^{(j+1)\Delta t}
  dt  \langle h[\mu_t^k], \varphi(t,x) \rangle  
 \Big| \leq Q \|\varphi\|_{\infty} \|\mu_t- \mu_t^k\|  \underset{k\to \infty}{\longrightarrow }0,
 \end{aligned}
\end{equation*}

and 
\begin{equation*}
\begin{aligned}
 \Big|
 &\dst\int_0^1 dt \langle  \mu_t ,  v[\mu_t]. \nabla \varphi(t,x)\rangle 
-\sum_{j=0}^{2^k-1}\dst\int_{j \Delta t}^{(j+1)\Delta t}
  dt  \langle  \mu_t^k ,  v[\mu_{j\Delta t}^k]. \nabla \varphi(t,x)\rangle 
 \Big|
\\
&\leq 
 \Big| \sum_{j=0}^{2^k-1}\dst\int_{j \Delta t}^{(j+1)\Delta t}
  dt  \langle  \mu_t^k- \mu_t ,  v[\mu_{j\Delta t}^k]. \nabla \varphi(t,x)\rangle 
 \Big|
 +  \Big| \sum_{j=0}^{2^k-1}\dst\int_{j \Delta t}^{(j+1)\Delta t}
  dt  \langle  \mu_t^k , ( v[\mu_{j\Delta t}^k] - v[\mu_t^k]). \nabla \varphi(t,x)\rangle 
 \Big|
 \\
 &+  \Big| \sum_{j=0}^{2^k-1}\dst\int_{j \Delta t}^{(j+1)\Delta t}
  dt  \langle  \mu_t^k , ( v[\mu_t] - v[\mu_t^k]). \nabla \varphi(t,x)\rangle 
 \Big| \\
& \leq 
\|\nabla\varphi\|_{\infty} \left( M \|\mu_t- \mu_t^k\|  
+  LM(P + |\mu_0|)  2^{-(k+1)}
+  (P + |\mu_0|)L \|\mu_t- \mu_t^k\|\right)
  \underset{k\to \infty}{\longrightarrow }0.
 \end{aligned}
\end{equation*}
This proves \eqref{eq:goal}.

{\bf Step 2. Any weak solution to \eqref{transport_equation} is Lipschitz in time.}
In this step, we prove that any weak solution \rev{in the sense of Definition \ref{def:mv_solution}} to the transport equation \eqref{transport_equation}
is Lipschitz with respect to time, since it satisfies 
\begin{equation}
\label{Lipschitz_estimate}
 \|\mu_{t+\tau} - \mu_t\|^{a,b} \leq L_1 \tau, \qquad t\geq 0, \; \tau \geq 0,
\end{equation}
with $L_1= P+bM(P(t+\tau)+|\mu_0|)$. To do so, we consider a solution $\mu_t$ to \eqref{transport_equation}. We define 
the vector field $w(t,x):=v[\mu_t](x)$ and the signed measure $b_t=h[\mu_t]$.
The vector field $w$ is uniformly Lipschitz and uniformly bounded with respect to $x$, since $v$ is so. The field $w$ is also measurable in time, since by definition, $\mu_t$ is continuous in time.
Then, $\mu_t$ is the unique solution of 
 \begin{equation}
 \label{transport_equation2}
  \p_t \mu_t(x) + \div.(w(t,x) \mu_t(x)) = b_t(x), \qquad   \mu_{|t=0}(x)=\mu_0(x).
 \end{equation}
Uniqueness of the solution of the linear equation \eqref{transport_equation2} is a direct consequence of 
Lemma \ref{lemma:characteristics}. 
Moreover, the scheme presented in Section \ref{s-scheme} can be rewritten for the vector field $w$ in which dependence with respect to time is added 
and dependence with respect to the measure is dropped. Thus, the unique solution $\mu_t$ to \eqref{transport_equation2} can be obtained as the limit 
of this scheme.
We have for $k\geq 0$ the following estimate
\begin{equation*}
 \| \mu_{t+\tau} -\mu_t\|^{a,b} \leq \|\mu_t - \mu_t^k\|^{a,b}+  \|\mu_t^k - \mu_{t+\tau}^k\|^{a,b} + \|\mu_{t+\tau}^k - \mu_{t+\tau}\|^{a,b},
\end{equation*}
where $\mu_t^k$ is given by the scheme.
The first and third terms can be rendered as small as desired for $k\geq k_0$ large enough, independent on $t,\tau$. For $\ell:= \min\{i \in \{1,\dots,2^k\}, \;  t \leq \frac{i}{2^k}\}$, $j:= \min\{i \in \{1,\dots,2^k\}, \; ,t+\tau \leq \frac{i}{2^k}\}$
 with the notations of the scheme, it holds
\begin{equation*}
\begin{aligned}
   \|\mu_{t+\tau}^k - \mu_t^k\|^{a,b}= \|m_j^k-m_{\ell}^k\|^{a,b}
   &=  \|\sum_{i=\ell}^{j-1} (m_{i+1}^k -m_i^k) \|^{a,b}
 = \|\sum_{i=\ell}^{j-1} (\phi_{\Delta t}^{v[m_i^k]}   \# m_i^k + \Delta t h[m_i^k] -m_i^k) \|^{a,b}
 \\
 &\leq \sum_{i=\ell}^{j-1} \|\phi_{\Delta t}^{v[m_i^k]}   \# m_i^k -m_i^k \|^{a,b} + \Delta t \|\sum_{i=\ell}^{j-1} h[m_i^k] \|^{a,b}.
 \end{aligned}
   \end{equation*}
   
Using Lemma \ref{lemma:flow} and \eqref{bounded_mass}, it holds
\begin{equation}
  \sum_{i=\ell}^{j-1} \| \phi_{\Delta t}^{v[m_i^k]}   \# m_i^k -m_i^k \|^{a,b}\leq \dfrac{j-\ell}{2^k} b M  (P(t+\tau) + |\mu_0|)
  \leq bM(P(t+\tau)+|\mu_0|)\tau + \dfrac{bM(P(t+\tau)+|\mu_0|)}{2^k}.
  \label{e-1}
\end{equation}
Using  \ref{H3}, we have
\begin{equation}
 \Delta t \|\sum_{i=\ell}^{j-1} h[m_i^k] \|^{a,b} \leq \dfrac{j-\ell}{2^k}P \leq P(t+\tau) \left(\tau +\dfrac{1}{2^k}\right),\label{e-2}
\end{equation}
Merging \eqref{e-1}-\eqref{e-2} and letting $k\to\infty$, we recover \eqref{Lipschitz_estimate}.

   {\bf Step 3. Any weak solution to \eqref{transport_equation} satisfies} the operator splitting estimate:
   \begin{equation}
   \label{stab_estimate}
    \|\mu_{t+ \tau} - (\phi_{\tau}^{v[\mu_t]} \# \mu_t + \tau h[\mu_t])\|^{a,b} \leq K_1 \tau^2,
   \end{equation}
 for some $K_1>0$. Indeed, let us consider a solution $\mu_t$ to \eqref{transport_equation}.
As in the previous step, $\mu_t$ is the unique solution to \eqref{transport_equation2}, and thus 
it can be obtained as the limit of the sequence provided by the scheme.
With the notations used in Step 2 and using Lemma \ref{lemma:flow}
\begin{equation*}
\begin{aligned}
\|\mu_{t+\tau}- (\phi_{\tau}^{v[\mu_t]} \# \mu_t + \tau h[\mu_t])\|^{a,b} 
&\leq 
\|\mu_{t+\tau}- \mu_{t+\tau}^k\|^{a,b} 
+  \|\mu_{t+\tau}^k- (\phi_{\tau}^{v[\mu_t^k]} \# \mu_t^k + \tau h[\mu_t^k])\|^{a,b} \\
&+ \tau \|h[\mu_t^k]-h[\mu_t]\|^{a,b}+ \| \phi_{\tau}^{v[\mu_t]} \# \mu_t- \phi_{\tau}^{v[\mu_t^k]}  \#\mu_t^k \|^{a,b}. \\
\end{aligned}
\end{equation*}

The first, third and fourth terms can be rendered as small as needed for $k$ sufficiently large, independently on $\tau$. We focus then on the second term. 
Assume for simplicity that $t=\ell\Delta t$ and $t+\tau = (\ell + n)\Delta t$, we have
\begin{equation*}
\begin{aligned}
 \|\mu_{t+\tau}^k- (\phi_{\tau}^{v[\mu_t]} \# \mu_t^k + \tau h[\mu_t])\|^{a,b}
=\|m_{\ell+n}^k- (\phi_{n\Delta t}^{v[m_{\ell}^k]} \# m_{\ell}^k + n\Delta t \;h[m_{\ell}^k])\|^{a,b}.
\end{aligned}
\end{equation*}
For $n=2$, we have
\begin{equation*}
 \begin{aligned}
 & \|m_{\ell+2}^k- (\phi_{2\Delta t}^{v[m_{\ell}^k]} \# m_{\ell}^k + 2\Delta t \;h[m_{\ell}^k])\|^{a,b}
  =\|\phi_{\Delta t}^{v[m_{\ell+1}^k]} \# m_{\ell+1}^k +\Delta t \;h[m_{\ell+1}^k]-\phi_{\Delta t}^{v[m_{\ell}^k]}\#\phi_{\Delta t}^{v[m_{\ell}^k]} \# m_{\ell}^k - 2\Delta t \;h[m_{\ell}^k]\|^{a,b}
 \\&=\|\phi_{\Delta t}^{v[m_{\ell+1}^k]} \#\left(\phi_{\Delta t}^{v[m_{\ell}^k]}\# m_{\ell}^k + \Delta t \; h[m_{\ell}^k]\right)  +\Delta t \;h[m_{\ell+1}^k]-\phi_{\Delta t}^{v[m_{\ell}^k]}\#\phi_{\Delta t}^{v[m_{\ell}^k]} \# m_{\ell}^k - 2\Delta t \;h[m_{\ell}^k]\|^{a,b}
  \\&=\|\phi_{\Delta t}^{v[m_{\ell+1}^k]} \#\phi_{\Delta t}^{v[m_{\ell}^k]}\# m_{\ell}^k + \Delta t \;\phi_{\Delta t}^{v[m_{\ell+1}^k]} \# h[m_{\ell}^k]+\Delta t \;h[m_{\ell+1}^k]-\phi_{\Delta t}^{v[m_{\ell}^k]}\#\phi_{\Delta t}^{v[m_{\ell}^k]} \# m_{\ell}^k - 2\Delta t \;h[m_{\ell}^k]\|^{a,b}
 \\&\leq \|\phi_{\Delta t}^{v[m_{\ell+1}^k]} \#\phi_{\Delta t}^{v[m_{\ell}^k]}\# m_{\ell}^k -\phi_{\Delta t}^{v[m_{\ell}^k]}\#\phi_{\Delta t}^{v[m_{\ell}^k]} \# m_{\ell}^k \|^{a,b}
 +\Delta t \|  \phi_{\Delta t}^{v[m_{\ell+1}^k]} \# h[m_{\ell}^k]+h[m_{\ell+1}^k]- 2h[m_{\ell}^k]\|^{a,b}
 \end{aligned}
\end{equation*}
Using Step 2, we have $\|m_{\ell+n}^k-m_{\ell}^k\| \leq L_1 n\Delta t$. Then, using Lemma \ref{lemma:flow}
\begin{equation*}
 \begin{aligned}
 & \|m_{\ell+2}^k- (\phi_{2\Delta t}^{v[m_{\ell}^k]} \# m_{\ell}^k + 2\Delta t \;h[m_{\ell}^k])\|^{a,b}\leq
 C\Delta t^2
 \end{aligned}
\end{equation*}

By induction on $i= 1\dots n$, it then holds
\begin{equation*}
 \begin{aligned}
 & \|m_{\ell+n}^k- (\phi_{n\Delta t}^{v[m_{\ell}^k]} \# m_{\ell}^k + n\Delta t \;h[m_{\ell}^k])\|^{a,b}\leq
C (n\Delta t)^2, \end{aligned}
\end{equation*}
and \eqref{stab_estimate} follows.

   {\bf Step 4. Uniqueness of the solution to  \eqref{transport_equation} and continuous dependence.}
  Assume that $\mu_t$ and $\nu_t$ are two solutions to \eqref{transport_equation} with initial condition $\mu_0,\nu_0$, respectively. Define $\eps(t) := \|\mu_t- \nu_t\|^{a,b}$.
  We denote 
  
\begin{equation*}
  R_{\mu}(t,\tau)=\mu_{t+ \tau} - (\phi_{\tau}^{v[\mu_t]} \# \mu_t + \tau h[\mu_t]), \quad 
    R_{\nu}(t,\tau)=\nu_{t+ \tau} - (\phi_{\tau}^{v[\nu_t]} \# \nu_t + \tau h[\nu_t]).
 \end{equation*}
 
   Using Lemma \ref{lemma:flow} and Step 3, and $e^{L\tau}\leq 1+2L\tau$ for $0\leq L\tau\leq \ln(2),$ we have that
   $\eps(t)$ is Lipschitz and it satisfies
\begin{eqnarray*}
    \eps(t+\tau)&=& \|\mu_{t+\tau} - \nu_{t+\tau}\|^{a,b}=\|\phi_{\tau}^{v[\mu_t]}\#\mu_t +\tau h[\mu_t] +R_{\mu}(t,\tau)-\phi_{\tau}^{v[\nu_t]}\#\nu_t -\tau h[\nu_t] -R_{\nu}(t,\tau) \|^{a,b}\\
    &\leq& \|\phi_{\tau}^{v[\mu_t]}\#\mu_t -\phi_{\tau}^{v[\nu_t]}\#\nu_t \|^{a,b}  +\tau\| h[\mu_t]-h[\nu_t] \|^{a,b} + \| R_{\mu}(t,\tau)\|^{a,b} + \|R_{\nu}(t,\tau) \|^{a,b}\\
    &\leq& e^{L\tau} \|\mu_t -\nu_t \|^{a,b}+ b(P+|\mu_0|)\frac{e^{L\tau}-1}{L}\|v[\mu_t] -v[\nu_t] \|_{\C^0}
    +\tau Q\|\mu_t -\nu_t \|^{a,b}+2K_1\tau^2    \\
   &\leq &\left(e^{L\tau} + b(P+\MT{\min\{|\mu_0|,|\nu_0|\}})2\tau K+\tau Q\right) \|\mu_t-\nu_t\|^{a,b} +2 K_1 \tau^2 \\
   &\leq& \left( 1+ \tau(2L+2 b K(P+\MT{\min\{|\mu_0|,|\nu_0|\}}) + Q)\right) \|\mu_t-\nu_t\|^{a,b} +2 K_1 \tau^2,
   \end{eqnarray*}

which is
\begin{equation}
 \dfrac{\eps(t+\tau)-\eps(t)}{\tau} \leq M \eps(t) +2K_1\tau , \qquad t>0,\; \tau \leq \dfrac{\ln(2)}{L},\quad M= 2L+2K(P+\MT{\min\{|\mu_0|,|\nu_0|\}}) + Q.
\end{equation}
Letting $\tau$ go to zero, we deduce $ \eps'(t) \leq M \eps(t)$ almost everywhere. Then, $\eps(t)\leq \eps(0) \exp(Mt)$, that is continuous dependence with respect to the initial data.

Moreover, if $\mu_0=\nu_0$, then $\eps(0)=0$, thus $\eps(t)=0$ for all $t$. Since $\|.\|^{a,b}$ is a norm, this implies $\mu_t=\nu_t$ for all $t$, that is uniqueness of the solution.

\bibliographystyle{plain}
 \bibliography{bibli19}


\end{document}